\documentclass{amsart}
\usepackage{amssymb}

\usepackage{amscd}
\usepackage{thmdefs}
\usepackage[dvips]{epsfig}

\input{tcilatex}
\theoremstyle{definition}
\theoremstyle{remark}
\numberwithin{equation}{section}

\input tcilatex

\begin{document}
\title[Toeplitz Operators Over Infinite Graphs]{Toeplitz Operators in Hilbert Space Over Infinite Graphs}
\author{Ilwoo Cho and Palle E. T. Jorgensen}
\address{St. Ambrose Univ., Dept. of Math., 518 W. Locust St., Davenport, Iowa,
52803, U. S. A. / Univ. of Iowa, Dept. of Math., 14 McLean Hall, Iowa City,
Iowa, 52242, U. S. A. }
\email{chowoo@sau.edu / jorgen@math.uiowa.edu}
\thanks{The second named author is supported by the U. S. National Science
Foundation.}
\date{Aug., 2010}
\dedicatory{}
\thanks{}
\maketitle

\begin{abstract}
Associated with a given graph $G,$ typically an infinite tree, and motivated
by applications, we introduce two families of operators in a Hilbert space $%
H_{G}$ induced by $G.$ To realize the Hilbert space, we first develop some
representation theory. We obtain the first family of operators on $H_{G}$ by
an extension of the more familiar case of groups: free representations of
the group-algebra. Because of their classical counter parts, we call the
operators in our first family, graph operators; and the second Toeplitz
operators. We focus on the interconnections between the two families. We
introduce and study graph operators in two steps: first, starting with a
fixed graph $G,$ we introduce a groupoid over $G$; and from this, the
groupoid von Neumann algebra $M_{G}.$ Our graph operators will then be
finitely supported elements of $M_{G}.$
\end{abstract}

\strut

\section{Introduction}

Before, starting our problem, we open with a historical comment, and a
comparison between the case of groups and graphs. In a number of recent
papers there have been a variety of different approaches to introducing
algebras of operators in Hilbert space (for a sample, see the papers cited
below). A number of these ideas are motivated by what works for groups,
i.e., starting with the group algebra, and then build representations of it.
Each representation serves some purpose, or is dictated by an application,
for example to harmonic analysis or to quantum mechanics. More than half a
century ago, von Neumann introduced the ring of operators (now called von
Neumann algebras) generated by the free group $F_{n}$ with $n$-generators,
leading to non-hyperfinite factors $L(F_{n})$ (See [27]). While the
construction is simple enough, the questions are difficult. Now, for $F_{n},$
the natural Hilbert space is $l^{2}(F_{n}).$ Since a group acts on itself,
we get operators in $l^{2}(F_{n}),$ i.e., regular representation; and $%
L(F_{n})$ is simply the von Neumann algebra generated by the regular
representation.

Now, let $G$ be a countable directed graph, i.e., a system of vertices and
edges (with direction) subject to simple axioms, details below. It is
tempting, in the analysis of graphs, to mimic some of the constructions used
for groups. But a glance at the comments above and literature shows that
there are difficulties for graphs that do not arise in the case of groups. A
key idea we employ is in brief outline this: Starting with a graph $G,$ we
introduce first an ``enveloping'' groupoid $\Bbb{G}$ and a groupoid algebra $%
\mathcal{A}_{G}.$ We show that $\mathcal{A}_{G}$ contains a canonical
abelian subalgebra $\mathcal{D}_{G}$ (the letter $\mathcal{D}$ for
diagonal!) and a conditional expectation $E_{G}$ from $\mathcal{A}_{G}$ onto 
$\mathcal{D}_{G}.$ We are then able to apply Stinespring's theorem (See
[26]) to $E_{G}.$ The resulting representation is acting on the Stinespring
Hilbert space $H_{G},$ and this will be the starting point of our analysis.

For the benefit of the readers, we collect here some references: [26] the
paper by Stinespring. While there are several relevant papers about von
Neumann's construction, the following will do for our present purpose [27].
There is a diverse set of approaches to Hilbert space, operators, and
operator algebras, and we list here only a sample: [4], [5], [6], [8], [14],
[15], [23], [24], and [25], dealing with graph groupoid dynamical systems
and corresponding crossed product von Neumann algebras, graph Laplacian
Operators, and reproducing kernels, etc. In a different direction, there is
a large literature on graph $C^{*}$-algebras, see for example [20], [21],
and [22]. For relevant papers in graph theory proper, we cite [1], [2],
[16], [17], [18], and [19]; again just a small sample.

Starting with analysis on countable directed graphs $G,$ we introduce
Hilbert spaces $H_{G}$ and a family of weighted operators $T$ on $H_{G}.$
When the weights (which are called \emph{coefficients} later in context) are
chosen, $T$ is called a \emph{graph operator}. From its weights (or
coefficients), we define the support $Supp(T)$ of $T.$

Let $G$ be a countable directed graph. Then there exists a corresponding
algebraic structure $\Bbb{G}$, as a form of groupoid (e.g., see [4], [9],
and [10]). Such a groupoid $\Bbb{G}$ is called the \emph{graph groupoid of} $%
G.$ By constructing the ($C^{*}$- or von Neumann) operator algebra $\mathcal{%
A},$ generated by $\Bbb{G},$ we can study the elements of $\Bbb{G}$ (or $G$)
as operators in $\mathcal{A},$ under a suitable representation of $\Bbb{G}$
(e.g., see [4], [5], and [7]). Interestingly, every operator $x$ on an
arbitrary (separable countable dimensional) Hilbert space $H$ can generate
the \emph{corresponding graph} $G_{a}$, and the $C^{*}$-algebra $C^{*}(a)$,
generated by $a$, is characterized by the groupoid crossed product algebra $%
A_{q}$ $\times _{\alpha }$ $\Bbb{G}_{a}$ induced by the groupoid dynamical
system $(A_{q},$ $\Bbb{G}_{a},$ $\alpha ),$ whenever $a$ is polar-decomposed
by $aq,$ where $a$ is the partial isometry part of $x,$ and $q$ is the
positive part of $x$ (e.g., see [6]).

The main motivation of our study is the above close connection between
directed graphs and operators.

In [8], We defined the graph operators by the ``finitely'' supported
operators in the von Neumann algebra $M_{G}$ $=$ $\overline{\Bbb{C}[L(\Bbb{G}%
)]}^{w},$ generated by $\Bbb{G}$, acting on the Hilbert space $H_{G},$ where
a representation $(H_{G},$ $L)$ of $\Bbb{G}$ is the canonical representation
of $\Bbb{G},$ consisting of the Stinespring Hilbert space\thinspace $H_{G},$
and the canonical groupoid action $L$ (Also, see Section 2 below).

Self-adjointness, the unitary property, Hyponormality and Normality of graph
operators are characterized in [8]. This means that the spectral-property of
graph operators are characterized. These operator-theoretic properties are
characterized by the combinatorial data on supports and the analytic data on
coefficients of graph operators.

In this paper, we find connections between our graph operators and the
well-known Toeplitz operators.

A Futher point motivating our work is from the analysis of general classes
of graphs. For general discrete models, there is no obvious group and
therefore no Fourier duality abailable. Graphs typically are not endowed
with a group structure that invites any kind of Fourier duality. As a basis
for our harmonic analysis, we instead introduce a natural groupoid which
serves as a substitute.

In the body of our paper, we will be using freely tools from the operators
in Hilbert space, the theory of $C^{*}$-algebras and von Neumann algebras.
The reader may find the following background references helpful: [35], and
[36]. Recent relevant papers on graph analysis include: [2], [4], [5], [6],
[7], [28], [29], [30], [31], [32], [33], and [34].

\subsection{Overview}

A \emph{graph} is a set of objects called \emph{vertices} (or points or
nodes) connected by links called \emph{edges} (or lines). In a \emph{%
directed graph}, the two directions are counted as being distinct directed
edges (or arcs). A graph is depicted in a diagrammatic form as a set of dots
(for vertices), jointed by curves (for edges). Similarly, a directed graph
is depicted in a diagrammatic form as a set of dots jointed by arrowed
curves, where the arrows point the direction of the directed edges.

Recently, we have studied the operator-algebraic structures induced by
directed graphs. The key idea to study graph-depending operator algebras is
that: every directed graph $G$ induces its corresponding \emph{groupoid} $%
\Bbb{G},$ called the\emph{\ graph groupoid of} $G$. By considering this
algebraic structure $\Bbb{G},$ we can determine the groupoid actions $%
\lambda ,$ acting on Hilbert spaces $H$. i.e., we can have suitable
representations $(H,$ $\lambda )$ for $\Bbb{G}.$ And this guarantees the
existence of operator algebras $\mathcal{A}_{G}$ $=$ $\overline{\Bbb{C}%
[\lambda (\Bbb{G})]}^{w},$ generated by $\Bbb{G}$ (or induced by $G$), in
the operator algebras $B(H).$ Indeed, the operator algebras $\mathcal{A}_{G}$
are the groupoid topological ($C^{*}$- or $W^{*}$-)subalgebras of $B(H).$

It is interesting that each edge $e$ of $G$ assigns a partial isometry on $%
H, $ and every vertex $v$ of $G$ assigns a projection on $H$ (under various
different types of representations of $\Bbb{G}$). For the continuation of
our recent research, we will fix the \emph{canonical representation\ }$%
(H_{G},$ $L)$\emph{\ of} $\Bbb{G},$ and construct the corresponding von
Neumann algebra

\begin{center}
$M_{G}$ $=$ $\overline{\Bbb{C}[L(\Bbb{G})]}^{w}$ in $B(H_{G}),$
\end{center}

where $H_{G}$ is the \emph{graph Hilbert space} $l^{2}(\Bbb{G})$. This von
Neumann algebra $M_{G}$ is called \emph{the graph von Neumann algebra of} $%
G. $

In this paper, we are interested in certain elements $T$ of $M_{G}.$ Recall
that, by the definition of graph von Neumann algebras, if $T$ $\in $ $M_{G},$
then

\begin{center}
$T$ $=$ $\underset{w\in \Bbb{G}}{\sum }$ $t_{w}L_{w}$ with $t_{w}$ $\in $ $%
\Bbb{C}.$
\end{center}

Define the support $Supp(T)$ of $T$ by

\begin{center}
$Supp(T)$ $=$ $\{w$ $\in $ $\Bbb{G}$ $:$ $t_{w}$ $\neq $ $0\}.$
\end{center}

If $Supp(T)$ is finite, then we call such an operator $T$ a \emph{graph
operator}. If $T$ is a graph operator, then the quantities $t_{w}$, for $w$ $%
\in $ $Supp(T),$ will be called the\emph{\ coefficients of} $T.$

As we can see all graph operators are the (finite) linear sum of generating
operators $L_{w}$'s of $M_{G},$ for $w$ $\in $ $\Bbb{G}.$ i.e., they are the
operators generated by finite numbers of projections and partial isometries
on $H_{G}.$

\subsection{Motivation}

Let $G$ be a graph with its graph groupoid $\Bbb{G},$ and let $H_{G}$ be the
Stinespring (Hilbert) space in the sense of Section 2. i.e., it is
Hilbert-space isomorphic to the $l^{2}$-space $l^{2}(\Bbb{G})$ of $\Bbb{G}.$
Then this Hilbert space $H_{G}$ contains its subspace $H_{V}$ $=$ $l^{2}(V),$
where $V$ means the vertex set of $G.$ In fact,

\begin{center}
$H_{G}$ $=$ $H_{V}$ $\oplus $ $H_{F},$
\end{center}

for some subspace $H_{F}.$ Remark here that, if the edge set $E$ of $G$ is
nonempty, then the orthogonal complement $H_{F}$ of $H_{V}$ is nontrivial in 
$H_{G}.$

We remark here that if a graph $G$ is the $N$-regular tree $\mathcal{T}_{N},$
for $N$ $\in $ $\Bbb{N},$ then

\begin{center}
$H_{V}$ $\overset{\text{Hilbert}}{=}$ $l^{2}\left( \Bbb{N}^{\oplus N}\right) 
$
\end{center}

(See Section 4 below). For example, the $2$-regular tree $\mathcal{T}_{2}$
is a tree

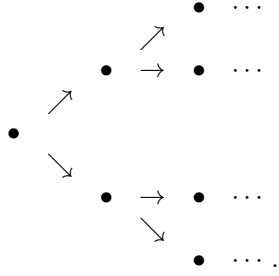
\strut

\begin{center}
$
\begin{array}{llllll}
&  &  &  & \bullet & \cdots \\ 
&  &  & \nearrow &  &  \\ 
&  & \bullet & \rightarrow & \bullet & \cdots \\ 
& \nearrow &  &  &  &  \\ 
\bullet &  &  &  &  &  \\ 
& \searrow &  &  &  &  \\ 
&  & \bullet & \rightarrow & \bullet & \cdots \\ 
&  &  & \searrow &  &  \\ 
&  &  &  & \bullet & \cdots .
\end{array}
$
\end{center}

\strut

Clearly, the 1-regular tree $\mathcal{T}_{1}$ is an infinite linear graph
with its root,

\strut

\begin{center}
$\bullet \longrightarrow \bullet \longrightarrow \bullet \longrightarrow
\cdot \cdot \cdot .$
\end{center}

\strut

In [8], we considered the unitarily equivalent infinite matrix $A_{e}$ on $%
l^{2}(\Bbb{N})$ of a graph operator $L_{e},$ for $e$ $\in $ $E_{1}$ of $%
\mathcal{T}_{1},$ represented on $H_{V}.$ Remark in this case that

\begin{center}
$H_{V}$ $\overset{\text{Hilbert}}{=}$ $l^{2}(\Bbb{N}),$
\end{center}

and

\begin{center}
$A_{e}$ $=$ $\left( 
\begin{array}{lllllll}
0 & 0 &  &  &  &  & 0 \\ 
& \ddots & \ddots &  &  &  &  \\ 
&  & 0 & 0 &  &  &  \\ 
&  &  & 1 & 1 &  &  \\ 
&  &  &  & 0 & 0 &  \\ 
&  &  &  &  & \ddots & \ddots \\ 
0 &  &  &  &  &  & 
\end{array}
\right) $
\end{center}

(Also, See Section 4). This shows that the sum of graph operators

\begin{center}
$\underset{e\in E_{1}}{\sum }$ $L_{e}$ in the von Neumann algebra $M_{%
\mathcal{T}_{1}}$
\end{center}

is represented by the Toeplitz operator

\begin{center}
$\left( 
\begin{array}{lllllll}
1 & 1 &  &  &  &  & 0 \\ 
& 1 & 1 &  &  &  &  \\ 
&  & 1 & 1 &  &  &  \\ 
&  &  & 1 & 1 &  &  \\ 
&  &  &  & 1 & 1 &  \\ 
&  &  &  &  & \ddots & \ddots \\ 
0 &  &  &  &  &  & 
\end{array}
\right) $
\end{center}

on $l^{2}(\Bbb{N})$ $=$ $H_{V}.$

Based on the above observation, we study the relation between graph
operators induced by $N$-regular trees $\mathcal{T}_{N}$ and Toeplitz
operators.

\section{Definitions and Background}

\strut In this section, we introduce the concepts precisely, and definitions
we will use.\strut \strut

\subsection{Graph Groupoids}

Let $G$ be a directed graph with its vertex set $V(G)$ and its edge set $%
E(G).$ Let $e$ $\in $ $E(G)$ be an edge connecting a vertex $v_{1}$ to a
vertex $v_{2}.$ Then we write $e$ $=$ $v_{1}$ $e$ $v_{2},$ for emphasizing
the initial vertex $v_{1}$ of $e$ and the terminal vertex $v_{2}$ of $e.$

For a fixed graph $G,$ we can define the oppositely directed graph $G^{-1},$
with $V(G^{-1})$ $=$ $V(G)$ and $E(G^{-1})$ $=$ $\{e^{-1}$ $:$ $e$ $\in $ $%
E(G)\},$ where each element $e^{-1}$ of $E(G^{-1})$ satisfies that

\begin{center}
$e$ $=$ $v_{1}$ $e$ $v_{2}$ in $E(G)$, with $v_{1},$ $v_{2}$ $\in $ $V(G),$
\end{center}

if and only if

\begin{center}
$e^{-1}$ $=$ $v_{2}$ $e^{-1}$ $v_{1},$ in $E(G^{-1}).$
\end{center}

This opposite directed edge $e^{-1}$ $\in $ $E(G^{-1})$ of $e$ $\in $ $E(G)$
is called the \emph{shadow of} $e.$ Also, this new graph $G^{-1}$, induced
by $G,$ is said to be the \emph{shadow of} $G.$ It is clear that $%
(G^{-1})^{-1}$ $=$ $G.$\strut

Define the \emph{shadowed graph} $\widehat{G}$ of $G$ by a directed graph
with its vertex set

\begin{center}
$V(\widehat{G})$ $=$ $V(G)$ $=$ $V(G^{-1})$
\end{center}

and its edge set

\begin{center}
$E(\widehat{G})$ $=$ $E(G)$ $\cup $ $E(G^{-1})$,
\end{center}

where $G^{-1}$ is the \emph{shadow} of $G$.

We say that two edges $e_{1}$ $=$ $v_{1}$ $e_{1}$ $v_{1}^{\prime }$ and $%
e_{2}$ $=$ $v_{2}$ $e_{2}$ $v_{2}^{\prime }$ are \emph{admissible}, if $%
v_{1}^{\prime }$ $=$ $v_{2},$ equivalently, the finite path $e_{1}$ $e_{2}$
is well-defined on $\widehat{G}.$ Similarly, if $w_{1}$ and $w_{2}$ are
finite paths on $G,$ then we say $w_{1}$ and $w_{2}$ are \emph{admissible},
if $w_{1}$ $w_{2}$ is a well-defined finite path on $G,$ too. Similar to the
edge case, if a finite path $w$ has its initial vertex $v$ and its terminal
vertex $v^{\prime },$ then we write $w$ $=$ $v_{1}$ $w$ $v_{2}.$ Notice that
every admissible finite path is a word in $E(\widehat{G}).$ Denote the set
of all finite path by $FP(\widehat{G}).$ Then $FP(\widehat{G})$ is the
subset of the set $E(\widehat{G})^{*},$ consisting of all finite words in $E(%
\widehat{G}).$

Suppose we take a part

\begin{center}
$
\begin{array}{lll}
& \bullet & \overset{e_{3}}{\longrightarrow }\cdot \cdot \cdot \\ 
& \uparrow & _{e_{2}} \\ 
\cdot \cdot \cdot \underset{e_{1}}{\longrightarrow } & \bullet & 
\end{array}
$
\end{center}

in a graph $G$ or in the shadowed graph $\widehat{G},$ where $e_{1},$ $%
e_{2}, $ $e_{3}$ are edges of $G,$ respectively of $\widehat{G}$. Then the
above admissibility shows that the edges $e_{1}$ and $e_{2}$ are admissible,
since we can obtain a finite path $e_{1}e_{2},$ however, the edges $e_{1}$
and $e_{3}$ are not admissible, since a finite path $e_{1}$ $e_{3}$ is
undefined.

We can construct the \emph{free semigroupoid} $\Bbb{F}^{+}(\widehat{G})$ of
the shadowed graph $\widehat{G},$ as the union of all vertices in $V(%
\widehat{G})$ $=$ $V(G)$ $=$ $V(G^{-1})$ and admissible words in $FP(%
\widehat{G}),$ equipped with its binary operation, the \emph{admissibility}$%
. $ Naturally, we assume that $\Bbb{F}^{+}\Bbb{(}\widehat{G})$ contains the 
\emph{empty word} $\emptyset ,$ as the representative of all undefined (or
non-admissible) finite words in $E(\widehat{G})$.

Remark that some free semigroupoid $\Bbb{F}^{+}\Bbb{(}\widehat{G})$ of $%
\widehat{G}$ does not contain the empty word; for instance, if a graph $G$
is a one-vertex-multi-edge graph, then the shadowed graph $\widehat{G}$ of $%
G $ is also a one-vertex-multi-edge graph too, and hence its free
semigroupoid $\Bbb{F}^{+}(\widehat{G})$ does not have the empty word.
However, in general, if $\left| V(G)\right| $ $>$ $1,$ then $\Bbb{F}^{+}(%
\widehat{G})$ always contain the empty word. Thus, if there is no confusion,
we always assume the empty word $\emptyset $ is contained in the free
semigroupoid $\Bbb{F}^{+}(\widehat{G})$ of $\widehat{G}.$

\begin{definition}
By defining the \emph{reduction }(RR) on $\Bbb{F}^{+}(\widehat{G}),$ we
define the graph groupoid $\Bbb{G}$ of a given graph $G,$ by the subset of $%
\Bbb{F}^{+}(\widehat{G}),$ consisting of all ``reduced'' finite paths on $%
\widehat{G},$ with the inherited admissibility on $\Bbb{F}^{+}(\widehat{G})$
under (RR), where the \emph{reduction} (RR) on $\Bbb{G}$ is\ \strut as
follows:

(RR)\qquad $\qquad \qquad \qquad w$ $w^{-1}$ $=$ $v$ and $w^{-1}w$ $=$ $%
v^{\prime },$\strut

for all $w$ $=$ $v$ $w$ $v^{\prime }$ $\in $ $\Bbb{G},$ with $v,$ $v^{\prime
}$ $\in $ $V(\widehat{G}).$
\end{definition}

Such a graph groupoid $\Bbb{G}$ is indeed a categorial groupoid with its
base $V(\widehat{G})$ (See \textbf{Appendix A}).

\subsection{Canonical Representation of Graph Groupoids}

\strut Let $G$ be a given countable connected directed graph with its graph
groupoid $\Bbb{G}.$ Then we can define the (pure algebraic) algebra $%
\mathcal{A}_{G}$ of $\Bbb{G}$ by a vector space over $\Bbb{C},$ consisting
of all linear combinations of elements of $\Bbb{G},$ i.e.,

\begin{center}
$\mathcal{A}_{G}$ $\overset{def}{=}$ $\Bbb{C}$ $\cup $ $\left( \underset{k=1%
}{\overset{\infty }{\cup }}\left\{ \sum_{j=1}^{k}t_{j}w_{j}\left| 
\begin{array}{c}
w_{j}\in \Bbb{G},\text{ }t_{j}\in \Bbb{C}, \\ 
j=1,...,k
\end{array}
\right. \right\} \right) ,$
\end{center}

under the usual addition ($+$), and the multiplication ($\cdot $), dictated
by the admissibility on $\Bbb{G}.$ Define now a unary operation ($*$) on $%
\mathcal{A}_{G}$ by

\begin{center}
$\sum_{j=1}^{k}$ $t_{j}$ $w_{j}$ $\in $ $\mathcal{A}_{G}$ $\longmapsto $ $%
\sum_{j=1}^{k}$ $\overline{t_{j}}$ $w_{j}^{-1}$ $\in $ $\mathcal{A}_{G},$
\end{center}

where $\overline{z}$ means the conjugate of $z,$ for all $z$ $\in $ $\Bbb{C}%
, $ and of course $w^{-1}$ means the shadow of $w,$ for all $w$ $\in $ $\Bbb{%
G}.$ We call this unary operation ($*$), the \emph{adjoint} (or the \emph{%
shadow}) on $\mathcal{A}_{G}.$ Then the vector space $\mathcal{A}_{G},$
equipped with the adjoint ($*$), is a well-defined (algebraic) $*$-algebra.

Now, define a $*$-subalgebra $\mathcal{D}_{G}$ of $\mathcal{A}_{G}$ by

\begin{center}
$\mathcal{D}_{G}$ $\overset{def}{=}$ $\Bbb{C}$ $\cup $ $\left( \underset{k=1%
}{\overset{\infty }{\cup }}\left\{ \sum_{j=1}^{n}t_{j}\text{ }v_{j}\left| 
\begin{array}{c}
v_{j}\in V(\widehat{G}),\text{ }t_{j}\in \Bbb{C}, \\ 
j=1,...,k
\end{array}
\right. \right\} \right) .$
\end{center}

This $*$-algebra $\mathcal{D}_{G}$ acts like the diagonal of $\mathcal{A}%
_{G},$ so we call $\mathcal{D}_{G},$ the \emph{diagonal} ($*$-)\emph{%
subalgebra of} $\mathcal{A}_{G}.$

\subsubsection{The Hilbert Space $H_{G}$}

Below, we identify the canonical Hilbert space $H_{G}.$ The algebra $%
\mathcal{A}_{G}$ is represented by bounded linear operators acting on $%
H_{G}. $ The representation is induced by the canonical conditional
expectation, via the Stinespring construction (e.g., see [14]).

We can construct a (algebraic $*$-)conditional expectation

\begin{center}
$E$ $:$ $\mathcal{A}_{G}$ $\rightarrow $ $\mathcal{D}_{G}$
\end{center}

by

(2.2.1)

\begin{center}
$E\left( \underset{w\in X}{\sum }\text{ }t_{w}w\right) $ $\overset{def}{=}$ $%
\underset{v\in X\cap V(\widehat{G})}{\sum }$ $t_{v}$ $v,$
\end{center}

for all $\underset{w\in X}{\sum }$ $t_{w}w$ $\in $ $\mathcal{A}_{G},$ where $%
X$ means a finite subset of $\Bbb{G}.$

Since the conditional expectation $F$ is completely positive under a
suitable topology on $\mathcal{A}_{G}$, we may apply the Stinespring's
construction. i.e., the diagonal subalgebra $\mathcal{D}_{G}$ is represented
as the $l^{2}$-space, $l^{2}(V(\widehat{G})),$ by the concatenation. Then we
can obtain the Hilbert space $H_{G},$

\begin{center}
$H_{G}$ $\overset{def}{=}$ the Stinespring space of $\mathcal{A}_{G}$ over $%
\mathcal{D}_{G},$ by $F,$
\end{center}

containing $l^{2}(V(\widehat{G})).$ i.e., if $\pi _{(E,\mathcal{D}_{G})}$ is
the Stinespring representation of $\mathcal{A}_{G},$ acting on $l^{2}(V(%
\widehat{G})),$ then

\begin{center}
$H_{G}$ $=$ $\pi _{(E,\mathcal{D}_{G})}\left( \mathcal{A}_{G}\right) .\strut 
$
\end{center}

This Stinespring space $H_{G}$ is the Hilbert space with its \emph{inner
product} $<,>_{G}$ satisfying that:

\begin{center}
$<h,$ $\pi _{(E,\mathcal{D}_{G})}(a)$ $k$ $>_{G}$ $=$ $<h,$ $E(a)$ $k$ $%
>_{2},$
\end{center}

for all $h,$ $k$ $\in $ $l^{2}(V(\widehat{G})),$ for all $a$ $\in $ $%
\mathcal{A}_{G},$ where $<,>_{2}$ is the inner product on $l^{2}(V(\widehat{G%
})).$\strut

i.e., The Stinespring space $H_{G}$ is the norm closure of $\mathcal{A}_{G},$
by the norm,

(2.2.2)

\begin{center}
$\left\| \sum_{j=1}^{n}\text{ }w_{i}\otimes h_{i}\right\| _{G}^{2}$ $=$ $%
\sum_{i=1}^{n}$ $\sum_{k=1}^{n}$ $<h_{i},$ $E(w_{i}^{*}w_{k})$ $h_{k}$ $%
>_{2},$
\end{center}

induced by the Stinespring inner product $<,>_{G}$ on $\mathcal{A}_{G},$ for
all $a_{i}$ $\in $ $\mathcal{A}_{G},$ $h_{i}$ $\in $ $l^{2}(V(\widehat{G})),$
for all $n$ $\in $ $\Bbb{N}.$

\begin{definition}
We call this Stinespring space $H_{G},$ the graph Hilbert space of $\Bbb{G}$
(or of $G$).
\end{definition}

\strut Denote the Hilbert space element $\pi _{(E,\mathcal{D}_{G})}(w)$ by $%
\xi _{w}$ in the graph Hilbert space $H_{G},$ for all $w$ $\in $ $\Bbb{G},$
with the identification,

\begin{center}
$\xi _{\emptyset }$ $=$ $0_{H_{G}},$ the zero vector in $H_{G},$
\end{center}

where $\emptyset $ is the empty word (if exists) of $\Bbb{G}.$ We can check
that the subset $\{\xi _{w}$ $:$ $w$ $\in $ $\Bbb{G}\}$ of $H_{G}$ satisfies
the following \emph{multiplication rule}:

\begin{center}
$\xi _{w_{1}}$ $\xi _{w_{2}}$ $=$ $\xi _{w_{1}w_{2}},$ on $H_{G},$
\end{center}

for all $w_{1},$ $w_{2}$ $\in $ $\Bbb{G}.$ Thus, we can define the \emph{%
canonical multiplication operators }$L_{w}$\emph{\ on} $H_{G}$, satisfying
that

\begin{center}
$L_{w}$ $\xi _{w^{\prime }}$ $\overset{def}{=}$ $\xi _{w}$ $\xi _{w^{\prime
}}$ $=$ $\xi _{ww^{\prime }},$
\end{center}

for all $w,$ $w^{\prime }$ $\in $ $\Bbb{G}.$ The existence of such
multiplication operators $L_{w}$'s guarantees the existence of a groupoid
action $L$ of $\Bbb{G},$ acting on $H_{G}$;

\begin{center}
$L$ $:$ $w$ $\in $ $\Bbb{G}$ $\longmapsto $ $L(w)$ $\overset{def}{=}$ $L_{w}$
$\in $ $B(H_{G}).$
\end{center}

This action $L$ of $\Bbb{G}$ is called the \emph{canonical groupoid action
of }$\Bbb{G}$\emph{\ on} $H_{G}.$

\subsubsection{\strut The Operators $L_{w}$}

Let $w$ and $w_{i}$ denote reduced finite paths in $FP_{r}(\widehat{G}),$
for $i$ $\in $ $\Bbb{N},$ equivalently, they are the reduced words in the
edge set $E(\widehat{G}),$ under the reduction (RR). Consider

(2.2.3)

\begin{center}
$L_{w}\left( \underset{i}{\sum }w_{i}\otimes h_{i}\right) $ $=$ $\underset{i%
}{\sum }$ $ww_{i}$ $\otimes $ $h_{i},$
\end{center}

for $h_{i}$ $\in $ $l^{2}(\Bbb{N})$. Here, the element $\underset{i}{\sum }$ 
$w_{i}$ $\otimes $ $h_{i}$ denotes a finite sum of tensors in $\mathcal{A}%
_{G}$. And $ww_{i}$ in (2.2.3) means concatenation of finite words. With the
conditional expectation $E$ $:$ $\mathcal{A}_{G}$ $\rightarrow $ $\mathcal{D}%
_{G}$ (See (2.2.1) above), we get the \emph{Stinespring representation} $%
(H_{G},$ $\pi _{(E,\mathcal{D}_{G})}),$ and the operators

$\pi _{(E,\mathcal{D}_{G})}(w)$ $:$ $\mathcal{H}_{G}$ $\rightarrow $ $%
\mathcal{H}_{G}$

obtained from (2.2.3) by passing to the quotient and completion as in
Definition 2.2. To simplify terminology, in the sequel, we will simply write 
$L_{w}$ for the operator $\pi _{(E,\mathcal{D}_{G})}(w).$

\subsubsection{\strut Graph von Neumann Algebras}

Let $G$, $\Bbb{G},$ and $H_{G}$ be given as above. And let $\{L_{w}$ $:$ $w$ 
$\in $ $\Bbb{G}\}$ the multiplication operators on $H_{G}$, where $L$ is the
canonical groupoid action of $\Bbb{G}.$

\begin{definition}
Let $G$ be a countable directed graph with its graph groupoid $\Bbb{G}.$ The
pair $(H_{G},$ $L)$ of the graph Hilbert space $H_{G}$ and the canonical
groupoid action $L$ of $\Bbb{G}$ is called the canonical representation of $%
\Bbb{G}$. The corresponding groupoid von Neumann algebra

\begin{center}
$M_{G}$ $\overset{def}{=}$ $\overline{\Bbb{C}[L(\Bbb{G})]}^{w},$
\end{center}

generated by $\Bbb{G}$ (equivalently, by $L(\Bbb{G})$ $=$ $\{L_{w}$ $:$ $w$ $%
\in $ $\Bbb{G}\}$), as a $W^{*}$-subalgebra of $B(H_{G})$, is called the
graph von Neumann algebra of $G.$
\end{definition}

\strut \strut We can check that the generating operators $L_{w}$'s of the
graph von Neumann algebra $M_{G}$ of $G$ satisfies that:

\begin{center}
$L_{w}^{*}$ $=$ $L_{w^{-1}},$ for all $w$ $\in $ $\Bbb{G},$
\end{center}

and

\begin{center}
$L_{w_{1}}L_{w_{2}}$ $=$ $L_{w_{1}w_{2}},$ for all $w_{1},$ $w_{2}$ $\in $ $%
\Bbb{G}.$\\[0pt]
\end{center}

It is easy to check that if $v$ is a vertex in $\Bbb{G},$ then the graph
operator $L_{v}$ is a projection, since

\begin{center}
$L_{v}^{*}$ $=$ $L_{v^{-1}}$ $=$ $L_{v}$ $=$ $L_{v^{2}}$ $=$ $L_{v}^{2}.$
\end{center}

Thus, by the reduction (RR) on $\Bbb{G},$ we can conclude that if $w$ is a
nonempty reduced finite path in $FP_{r}(\widehat{G}),$ then the operator $%
L_{w}$ is a partial isometry, since

\begin{center}
$L_{w}^{*}$ $L_{w}$ $=$ $L_{w^{-1}w}$,
\end{center}

and $w^{-1}w$ is a vertex, and hence $L_{w}^{*}L_{w}$ is a projection on $%
H_{G}.$\strut

\section{\strut Graph Operators}

In this section, we introduce graph operators and summarize the
operator-theoretical properties of graph operators obtained in [8]. These
results will be applied to characterize the operator-theoretic properties of
Toeplitz operators in Section 5.

Let $G$ be a graph with its graph groupoid $\Bbb{G},$ and let $M_{G}$ $=$ $%
\overline{\Bbb{C}[L(\Bbb{G})]}^{w}$ be the graph von Neumann algebra of $G$
in $B(H_{G}),$ where $(H_{G},$ $L)$ is the canonical representation of $\Bbb{%
G}.$ Since $M_{G}$ is a groupoid von Neumann algebra generated by $\Bbb{G},$
every element $T$ of $M_{G}$ satisfies the expansion,

\begin{center}
$T$ $=$ $\underset{w\in \Bbb{G}}{\sum }$ $t_{w}$ $L_{w},$ with $t_{w}$ $\in $
$\Bbb{C}.$
\end{center}

For the given operator $T$ $\in $ $M_{G},$ having the above expansion,
define the subset $Supp(T)$ of $\Bbb{G}$ by

\begin{center}
$Supp(T)$ $\overset{def}{=}$ $\{w$ $\in $ $\Bbb{G}$ $:$ $t_{w}$ $\neq $ $%
0\}. $
\end{center}

This subset $Supp(T)$ of $\Bbb{G}$ is called the \emph{support of} $T.$ And
the constants $t_{w}$'s, for $w$ $\in $ $Supp(T),$ are said to be the \emph{%
coefficients of }$T.$

\begin{definition}
Let $T$ be an element of the graph von Neumann algebra $M_{G}$ of a given
graph $G,$ and let $Supp(T)$ be the support of $T.$ If $Supp(T)$ is finite,
then we call the operator $T$, a graph operator. The graph operators $L_{w}$%
, generating $M_{G},$ for all $w$ $\in $ $\Bbb{G}$ $\setminus $ $\{\emptyset
\},$ are called the generating (graph) operators.
\end{definition}

i.e., the graph operators are the finitely supported operators on $H_{G}.$

In [8], we characterize the spectral-theoretical properties of graph
operators, in terms of their supports and coefficients. In this paper, we
concentrate on studying the connections between certain graph operators and
Toeplitz operators.

\section{Background for Main Results}

In this section, we consider the fundamental background of the main results
of this paper obtained in Sections 5 and 6. In Section 4.1, we discuss about
the decomposition of graph Hilbert spaces. We observe that whenever a graph
Hilbert space $H_{G}$ is given, there exists a (closed) subspace $H_{V},$
induced by all vertices of $G,$ such that

\begin{center}
$H_{G}$ $=$ $H_{V}$ $\oplus $ $H_{V}^{\perp },$
\end{center}

and $H_{V}^{\perp }$ is induced by all reduced finite paths of $G.$
Moreover, we will restrict our interests to the case where a given graph $G$
is a regular tree.

In Section 4.2, we consider generalized Toeplitz algebra $Toep(H)$ over an
arbitrary Hilbert space $H.$ In fact, the $C^{*}$-algebra $Toep(H)$ is
well-known, but we are particularly interested in the anti-$*$-isomorphic $%
C^{*}$-algebra $Toep^{r}(H)$ of $Toep(H).$

\subsection{Graph Operators Induced by Regular Trees}

In this section, we restrict our interests to the case where the given
graphs are regular trees $\mathcal{T}_{N}$, for $N$ $\in $ $\Bbb{N}.$ Notice
that the regular trees are \emph{simplicial}, in the sense that (i) they do
not allow loop-edges, and (ii) they do not have multi-edges, equivalently,
if there is an edge connecting two vertices, then there is no other edge
connecting those vertices. For instance, the following three graphs $G_{1},$ 
$G_{2}$, and $G_{3}$ are not simplicial, where

\strut

\begin{center}
$G_{1}$ $=$ \quad $\underset{\circlearrowleft }{\overset{v}{\bullet }}%
\longrightarrow \bullet ,$
\end{center}

\strut

\begin{center}
$G_{2}$ $=$ \quad $\underset{v_{1}}{\bullet }\rightrightarrows \underset{%
v_{2}}{\bullet }\leftarrow \bullet ,$
\end{center}

and

\begin{center}
$G_{3}$ $=$ \quad $\bullet \rightarrow \underset{\circlearrowleft }{\bullet }%
\leftleftarrows \bullet .$
\end{center}

\strut

Indeed, the graph $G_{1}$ has a loop-edge connecting from the vertex $v$ to
itself, and hence it is not simplicial; the graph $G_{2}$ contain two edges
connecting the vertex $v_{1}$ to the vertex $v_{2}$, and hence it is not
simplicial; the graph $G_{3}$ is not simplicial because it has both
loop-edge and multi-edges.

Since the regular trees $\mathcal{T}_{N}$ are simplicial, we can put the
suitable name (or indices) for the vertices. For the $N$-regular tree $%
\mathcal{T}_{N},$ we will put the name $1$ for the root of $\mathcal{T}_{N},$
and the $N$-vertices in the 1-st level of $\mathcal{T}_{N}$ have their names 
$11,$ $12,$ ..., $1N.$ And the $N^{2}$-vertices in the 2-nd level of $%
\mathcal{T}_{N}$ have their names $111,$ ..., $11N,$ $121,$ ..., $12N,$ ..., 
$1N1,$ ..., $1NN,$ etc. For instance, the $2$-regular tree $\mathcal{T}_{2}$
has its vertices with their indices as follows:

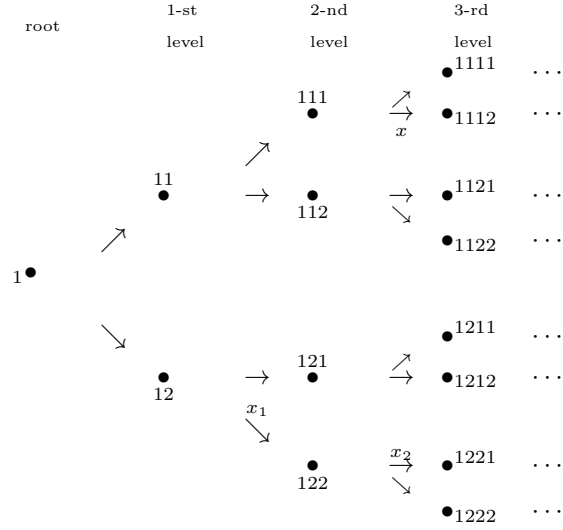
\strut

\begin{center}
$
\begin{array}{lllllllll}
\begin{array}{l}
\text{{\tiny root}}
\end{array}
&  & 
\begin{array}{l}
\text{{\tiny 1-st}} \\ 
\text{{\tiny level}}
\end{array}
&  & 
\begin{array}{l}
\text{{\tiny 2-nd}} \\ 
\text{{\tiny level}}
\end{array}
&  & 
\begin{array}{l}
\text{{\tiny 3-rd}} \\ 
\text{{\tiny level}}
\end{array}
&  &  \\ 
&  &  &  &  &  & \bullet ^{1111} & \cdots &  \\ 
&  &  &  & \overset{111}{\bullet } & \overset{\nearrow }{\underset{x}{%
\rightarrow }} & \bullet _{1112} & \cdots &  \\ 
&  &  & \nearrow &  &  &  &  &  \\ 
&  & \overset{11}{\bullet } & \rightarrow & \underset{112}{\bullet } & 
\underset{\searrow }{\rightarrow } & \bullet ^{1121} & \cdots &  \\ 
& \nearrow &  &  &  &  & \bullet _{1122} & \cdots &  \\ 
_{1}\bullet &  &  &  &  &  &  &  &  \\ 
&  &  &  &  &  &  &  &  \\ 
& \searrow &  &  &  &  & \bullet ^{1211} & \cdots &  \\ 
&  & \underset{12}{\bullet } & \rightarrow & \overset{121}{\bullet } & 
\overset{\nearrow }{\rightarrow } & \bullet _{1212} & \cdots &  \\ 
&  &  & \overset{x_{1}}{\searrow } &  &  &  &  &  \\ 
&  &  &  & \underset{122}{\bullet } & \underset{\searrow }{\overset{x_{2}}{%
\rightarrow }} & \bullet ^{1221} & \cdots &  \\ 
&  &  &  &  &  & \bullet _{1222} & \cdots & 
\end{array}
$
\end{center}

\strut

And each edge $e$ of $\mathcal{T}_{N}$ connecting the vertex $v_{1}$ to the
vertex $v_{2}$ can be denoted by the pair $(v_{1},$ $v_{2}),$ again by the
simpliciality of $\mathcal{T}_{N}.$ For instance, in $\mathcal{T}_{2},$ the
edge $x$ in the above figure is denoted by the pair $(111,$ $1112).$ Such a
pair notation does not fit for arbitrary graph case (in particular, where a
graph allows multi-edges). But, for simplicial graphs, this pair notation
works well.

Thus a length-$k$ finite path $w$ can be denoted by $(k$ $+$ $1)$-tuple of
passing vertices, for $k$ $\in $ $\Bbb{N}.$ For example, if $w$ $=$ $x_{1}$ $%
x_{2}$ in $\mathcal{T}_{2},$ where $x_{1}$ and $x_{2}$ are edges in the
above figure, then

\begin{center}
$x_{1}$ $=$ $(12,$ $122),$ $x_{2}$ $=$ $(122,$ $1221),$
\end{center}

and

\begin{center}
$w$ $=$ $x_{1}x_{2}$ $=$ $(12,$ $122,$ $1221).$
\end{center}

Clearly, if we have a finite path expressed by the $(k$ $+$ $1)$-tuple $w$ $%
= $ $(v_{1},$ $v_{2},$ ..., $v_{k+1}),$ then we can understand $w$ as a
length-$k$ finite path $w$ $=$ $e_{1}$ ... $e_{k},$ generated by the
admissible edges $e_{1},$ ..., $e_{k},$ where

\begin{center}
$e_{j}$ $=$ $(v_{j},$ $v_{j+1}),$ for all $j$ $=$ $1,$ ..., $k.$
\end{center}

For the given $N$-regular tree $\mathcal{T}_{N}$ (under the above setting on
vertices and edges) we can determine the graph groupoid $\Bbb{G}_{N}$ $=$ $%
\Bbb{G}_{\mathcal{T}_{N}},$ and the corresponding graph von Neumann algebra $%
M_{N}$ $=$ $M_{\mathcal{T}_{N}},$ for $N$ $\in $ $\Bbb{N}.$ We are
interested in graph operators in $M_{N}.$\strut

Let $\mathcal{T}_{N}$ be the $N$-regular tree with its graph groupoid $\Bbb{G%
}_{N},$ and let $H_{N}$ and $M_{N}$ be the corresponding graph Hilbert space
and the graph von Neumann algebra of $\mathcal{T}_{N},$ respectively, for $N$
$\in $ $\Bbb{N}.$ By the Stinespring construction, the graph Hilbert space $%
H_{N}$ has its subspace $H_{V},$

(4.1.1)

\begin{center}
$
\begin{array}{ll}
H_{V} & \,=l^{2}\left( V(\mathcal{T}_{N})\right) \\ 
& \overset{\text{Hilbert}}{=}\Bbb{C}^{\oplus \left| V(\mathcal{T}%
_{N})\right| }\overset{\text{Hilbert}}{=}\underset{v\in V(\mathcal{T}_{N})}{%
\oplus }\Bbb{C}\xi _{v},
\end{array}
$
\end{center}

where $V(\mathcal{T}_{N})$ $=$ $V\left( \widehat{\mathcal{T}_{N}}\right) $
is the vertex set of $\mathcal{T}_{N},$ where

\begin{center}
$\xi _{v}$ $=$ $\pi _{(E,\mathcal{D}_{\mathcal{T}_{N}})}(v),$ for all $v$ $%
\in $ $V(\mathcal{T}_{N})$
\end{center}

(See Section 2.2). We call the subspace $H_{V}$ of $H_{N},$ the \emph{vertex
space of} $\mathcal{T}_{N}.$ For convenience, let's denote $V(\mathcal{T}%
_{N})$ simply by $V_{N},$ for $N$ $\in $ $\Bbb{N}.$ Thus, the Hilbert space $%
H_{N}$ is decomposed by

\begin{center}
$H_{N}$ $=$ $H_{V}$ $\oplus $ $H_{FP},$
\end{center}

where

\begin{center}
$H_{FP}$ $=$ $H_{N}$ $\ominus $ $H_{V}$ $=$ $l^{2}\left( FP_{r}(\mathcal{T}%
_{N})\right) ,$
\end{center}

where $FP_{r}(\mathcal{T}_{N})$ is the reduced finite path set of $\Bbb{G}%
_{N}.$

Now, let's denote $E_{k}^{N}$ be the \emph{length-}$k$\emph{\ reduced finite
path set}, which is the subset of $FP_{r}\left( \widehat{\mathcal{T}_{N}}%
\right) $ consisting of all length-$k$ reduced finite paths on the shadowed
graph $\widehat{\mathcal{T}_{N}}$ of $\mathcal{T}_{N},$ for all $k$ $\in $ $%
\Bbb{N}.$ Clearly, the edge set $E\left( \widehat{\mathcal{T}_{N}}\right) $
of $\widehat{\mathcal{T}_{N}}$ is the set $E_{1}^{N},$ and the set $%
FP_{r}\left( \widehat{\mathcal{T}_{N}}\right) $ is partitioned by

\begin{center}
$FP_{r}\left( \widehat{\mathcal{T}_{N}}\right) $ $=$ $\underset{k=1}{%
\overset{\infty }{\sqcup }}$ $E_{k}^{N},$
\end{center}

set-theoretically, where $\sqcup $ means the disjoint union. So, the
subspace $H_{FP}$ of $H_{N}$ is Hilbert-space isomorphic to

\begin{center}
$H_{FP}$ $=$ $\underset{k=1}{\overset{\infty }{\oplus }}$ $\left( \underset{%
w\in E_{k}}{\oplus }\Bbb{C}\xi _{w}\right) ,$
\end{center}

whenever

\begin{center}
$\Bbb{C}\xi _{w}$ $\overset{\text{Hilbert}}{=}$ \ $\underset{k\text{-times}}{%
\underbrace{\Bbb{C}\otimes \cdot \cdot \cdot \otimes \Bbb{C}}}$ $\;=$ $\Bbb{C%
}^{\otimes k},$
\end{center}

for all $w$ $\in $ $E_{k},$ for all $k$ $\in $ $\Bbb{N},$ where

\begin{center}
$\xi _{w}$ $=$ $\pi _{(E,\mathcal{D}_{\mathcal{T}_{N}})}(w),$ for all $w$ $%
\in $ $FP_{r}\left( \widehat{\mathcal{T}_{N}}\right) .$
\end{center}

The above observation shows that the graph Hilbert space $H_{N}$ has its
orthonormal basis (or its Hilbert basis),

\begin{center}
$\{\xi _{w}$ $:$ $w$ $\in $ $\Bbb{G}_{N}$ $\setminus $ $\{\emptyset \}\}.$
\end{center}

Therefore, if we define the Hilbert space $l^{2}(\Bbb{G}_{N})$ by the $l^{2}$%
-space generated by $\Bbb{G}_{N}$ $\setminus $ $\{\emptyset \},$ more
precisely,

(4.1.2)

\begin{center}
$l^{2}(\Bbb{G}_{N})$ $\overset{def}{=}$ $\left( \underset{v\in V(\mathcal{T}%
_{N})}{\oplus }\text{ }\Bbb{C}\eta _{v}\right) $ $\oplus $ $\left( \underset{%
w\in FP_{r}\left( \widehat{\mathcal{T}_{N}}\right) }{\oplus }\text{ }\Bbb{C}%
\eta _{w}\right) ,$
\end{center}

with its Hilbert basis

\begin{center}
$\{\eta _{w}$ $:$ $w$ $\in $ $\Bbb{G}_{N}$ $\setminus $ $\{\emptyset \}\},$
\end{center}

then the graph Hilbert space $H_{G}$ and the Hilbert space $l^{2}(\Bbb{G}%
_{N})$ are Hilbert-space isomorphic

\begin{center}
$H_{G}$ $\overset{\text{Hilbert}}{=}$ $l^{2}(\Bbb{G}_{N}).$
\end{center}

\strut So, without loss of generality, we may consider our graph Hilbert
space $H_{N}$ (the Stinesping space) as $l^{2}(\Bbb{G}_{N}).$ In the
following context, we use $H_{G}$ and $l^{2}(\Bbb{G}_{N}),$ alternatively.

Remark that, in [4] and [7], we define the graph Hilbert space $H_{G}$ of a
given arbitrary countable directed graph $G$ by $l^{2}(\Bbb{G})$, where $%
\Bbb{G}$ is the graph groupoid of $G.$

\strut Now, consider the graph groupoid $\Bbb{G}_{N}$ of the $N$-regular
tree $\mathcal{T}_{N}$ more in detail.

Let $(v_{1},$ $v_{2})$ be an edge of $\mathcal{T}_{N}.$ Then its shadow has
its pair notation $(v_{2},$ $v_{1}).$ So, we can have

\begin{center}
$(v_{1},$ $v_{2})$ $(v_{2},$ $v_{1})$ $=$ $(v_{1},$ $v_{2},$ $v_{1})$ $=$ $%
v_{1},$
\end{center}

by the reduction (RR) on $\Bbb{G}_{N}.$ This means that, if we have a
``nonempty'' element

\begin{center}
$(v_{1},$ $v_{2},$ ..., \underline{\frame{$v_{j}$}$,$ ..., $v_{n},$ ..., 
\frame{$v_{j}$}}$,$ $v_{j+1},$ ..., $v_{k}),$
\end{center}

in $FP_{r}\left( \widehat{\mathcal{T}_{N}}\right) ,$ then it is reduced (and
hence identical) to a length-$(k$ $-$ $1)$ reduced finite path

\begin{center}
$(v_{1},$ ..., \underline{$v_{j}$}$,$ $v_{j+1},$ ..., $v_{k}),$
\end{center}

in $FP_{r}\left( \widehat{\mathcal{T}_{N}}\right) ,$ for $k$ $\in $ $\Bbb{N}$%
, where the length-0 reduced finite paths mean the vertices (i.e., where $k$ 
$=$ $1$).

\strut By operator theory, we can represent each element $T$ of the graph
von Neumann algebra $M_{N}$ on $H_{N}$. However, we are interested in the
representation of $T$ on the vertex space $H_{V}.$\strut

\subsection{\strut Toeplitz Algebras $Toep(H)$}

Let $H$ be an arbitrary Hilbert space, throughout this section. The \emph{%
Fock space }$\mathcal{F}_{H}$\emph{\ }$\overset{denote}{=}$\emph{\ }$%
\mathcal{F}(H)$\emph{\ over} $H$ is defined by a Hilbert space,

\begin{center}
$\mathcal{F}_{H}$ $\overset{def}{=}$ $\underset{n=0}{\overset{\infty }{%
\oplus }}$ $H^{\otimes n},$ with $H^{\otimes n}$ $=$ $\Bbb{C}\Omega $ $=$ $%
\Bbb{C},$
\end{center}

where $\Omega $ means the vacuum vector, where the direct sum $\oplus $ and
the tensor product $\otimes $ are all defined under the Hilbert (product)
topology.

Now, fix a Hilbert-space element $h$ of $H,$ and then define an operator $%
l_{h}$ on $\mathcal{F}_{H}$ by an operator satisfying

(4.2.1)

\begin{center}
$l_{h}(\Omega )$ $=$ $h,\qquad \qquad \qquad \qquad $ and

$l_{h}$ $:$ $\xi _{1}$ $\otimes $ ... $\otimes $ $\xi _{n}$ $\mapsto $ $h$ $%
\otimes $ $\xi _{1}$ $\otimes $ ... $\otimes $ $\xi _{n},$
\end{center}

for all $n$ $\in $ $\Bbb{N}.$ Then clearly, we can check that the adjoint $%
l_{h}^{*}$ of $l_{h}$ is an operator satisfying that:

(4.2.2)

\begin{center}
$l_{h}^{*}(\Omega )$ $=$ $0_{\mathcal{F}_{H}},$ \qquad \qquad \qquad \qquad
and

$l_{h}^{*}$ $:$ $\xi _{1}$ $\otimes $ $\xi _{2}$ $\otimes $ ... $\otimes $ $%
\xi _{n}$ $\mapsto $ $<h,$ $\xi _{1}$ $>_{H}$ $\xi _{2}$ $\otimes $ ... $%
\otimes $ $\xi _{n},$
\end{center}

for all $n$ $\in $ $\Bbb{N},$ where $0_{\mathcal{F}_{H}}$ is the zero vector
in $\mathcal{F}_{H},$ and $<,>_{H}$ means the inner product on the Hilbert
space $H.$

\begin{definition}
Let $l_{h}$ be an operator on the Fock space $\mathcal{F}_{H}$ over a
Hilbert space $H,$ defined in (4.2.1), for a fixed element $h$ $\in $ $H.$
Then it is called the (left) creation operator induced by $h.$ The adjoint $%
l_{h}^{*}$ of $l_{h},$ satisfying (4.2.2), is called the (left) annihilation
operator induced by $h.$
\end{definition}

\strut Remark that the operators

\begin{center}
$\{l_{h},$ $l_{h}^{*}$ $\in $ $B\left( \mathcal{F}_{H}\right) $ $:$ $h$ $\in 
$ $H\}$
\end{center}

satisfy the relation,

(4.2.3)

\begin{center}
$l_{h_{1}}^{*}$ $l_{h_{2}}$ $=$ $<h_{1},$ $h_{2}>_{H}$ $1_{\mathcal{F}_{H}},$
\end{center}

where $1_{\mathcal{F}_{H}}$ is the identity operator on $\mathcal{F}_{H},$
for all $h_{1},$ $h_{2}$ $\in $ $H.$

\begin{definition}
The $C^{*}$-subalgebra $Toep(H)$ of $B\left( \mathcal{F}_{H}\right) ,$
generated by the creation operators,

\begin{center}
$\{l_{h}$ $\in $ $B\left( \mathcal{F}_{H}\right) $ $:$ $h$ $\in $ $H\}$
\end{center}

is called the Toeplitz operator over $H.$ And the elements of $Toep(H)$ are
said to be (generalized) Toeplitz operators (over $H$). In particular, if $%
\dim H$ $=$ $1,$ then $Toep(H)$ is $*$-isomorphic to the classical Toeplitz
algebra $Toep$ in $B\left( l^{2}(\Bbb{N})\right) .$
\end{definition}

\strut Let $\mathcal{F}_{H}$ be the Fock space over $H$ given as above, also
let's fix a Hilbert-space element $h$ in $H.$ Now, we will define a new
operator $r_{h}$ induced by $h$ by an operator on $\mathcal{F}_{H}$
satisfying that

(4.2.4)

\begin{center}
$r_{h}(\Omega )$ $=$ $h,$ \qquad \qquad \qquad and

$r_{h}$ $:$ $\xi _{1}$ $\otimes $ ... $\otimes $ $\xi _{n}$ $\mapsto $ $\xi
_{1}$ $\otimes $ ... $\otimes $ $\xi _{n}$ $\otimes $ $h,$
\end{center}

\strut for all $n$ $\in $ $\Bbb{N}.$ Then, the adjoint $r_{h}^{*}$ of $r_{h}$
satisfies that

(4.2.5)

\begin{center}
$r_{h}^{*}(\Omega )$ $=$ $0_{\mathcal{F}_{H}}$ \qquad \qquad \qquad \qquad
and

$r_{h}^{*}$ $:$ $\xi _{1}$ $\otimes $ ... $\otimes $ $\xi _{n}$ $\otimes $ $%
\xi _{n+1}$ $\mapsto $ $\xi _{1}$ $\otimes $ ... $\otimes $ $\xi _{n}$ $<h,$ 
$\xi _{n+1}>_{H},$
\end{center}

for all $n$ $\in $ $\Bbb{N}.$

\begin{definition}
The operators $r_{h}$, satisfying (4.2.4), on $\mathcal{F}_{H}$ is called
the right creation operator induced by $h,$ and its adjoint $r_{h}^{*},$
satisfying (4.2.5), is called the right annihilation operator induced by $h.$
\end{definition}

The family of operators

\begin{center}
$\{r_{h},$ $r_{h}^{*}$ $\in $ $B\left( \mathcal{F}_{H}\right) $ $:$ $h$ $\in 
$ $H\}$
\end{center}

satisfies the relation,

(4.2.6)

\begin{center}
$r_{h_{1}}^{*}r_{h_{2}}$ $=$ $<h_{2},$ $h_{1}$ $>_{H}$ $1_{\mathcal{F}_{H}},$
\end{center}

for all $h_{1},$ $h_{2}$ $\in $ $H.$

\begin{definition}
The $C^{*}$-subalgebra $Toep^{R}(H)$ of $B\left( \mathcal{F}_{H}\right) ,$
generated by the right creation operators,

\begin{center}
$\{r_{h}$ $\in $ $B\left( \mathcal{F}_{N}\right) $ $:$ $h$ $\in $ $H\}$
\end{center}

is called the right Toeplitz algebra over $H.$ And the elements of $%
Toep^{R}(H)$ are said to be right (generalized) Toeplitz operators on $%
\mathcal{F}_{H}.$
\end{definition}

The following theorem shows the relation between the Toeplitz algebra $%
Toep(H),$ and the right Toeplitz algebra $Toep^{R}(H).$

\begin{theorem}
The Toeplitz algebra $Toep(H)$ over $H,$ and the right Toeplitz algebra $%
Toep^{R}(H)$ over $H$ are anti-$*$-isomorphic.
\end{theorem}

\begin{proof}
Recall that, by definition,

\begin{center}
$Toep(H)$ $=$ $C^{*}\left( \{l_{h}:h\in H\}\right) ,$
\end{center}

and

\begin{center}
$Toep^{R}(H)$ $=$ $C^{*}\left( \{r_{h}:h\in H\}\right) ,$
\end{center}

as $C^{*}$-subalgebras of $B\left( \mathcal{F}_{H}\right) ,$ where $l_{h}$
and $r_{h}$ are the left and right creation operators, respectively, for all 
$h$ $\in $ $H.$ Thus, we can define a map

$\Phi $ $:$ $Toep(H)$ $\rightarrow $ $Toep^{R}(H)$

by a generator-preserving linear transformation, satisfying

(4.2.7)

\begin{center}
$\Phi $ $:$ $\left\{ 
\begin{array}{l}
1_{L}\in Toep(H)\longmapsto 1_{R}\in Toep^{R}(H), \\ 
l_{h}\in Toep(H)\longmapsto r_{h}^{*}\in Toep^{R}(H), \\ 
l_{h}^{*}\in Toep(H)\longmapsto r_{h}^{**}=r_{h}\in Toep^{R}(H), \\ 
l_{h_{1}}^{q_{1}}l_{h_{2}}^{q_{2}}\in Toep(H)\mapsto \text{ }%
r_{h_{2}}^{q_{2}*}r_{h_{1}}^{q_{1}*},
\end{array}
\right. $
\end{center}

for all $h,$ $h_{1},$ $h_{2}$ $\in $ $H,$ and for all $q_{1},$ $q_{2}$ $\in $
$\{1,$ $*\},$ where $1_{L}$ is the identity element in $Toep(H),$ and $1_{R} 
$ is the identity element in $Toep^{R}(H).$

Since $\Phi $ is generator-preserving, it is bijective and bounded.
Moreover, it is not difficult to check $\Phi $ is isometric.

Observe now that by the 4-th condition of (4.2.7), the linear map $\Phi $ is
anti-multiplicative, i.e.,

(4.2.8)

\begin{center}
$\Phi \left( l_{h_{1}}^{q_{1}}l_{h_{2}}^{q_{2}}\right) $ $=$ $%
r_{h_{2}}^{q_{2}*}$ $r_{h_{1}}^{q_{1}*}$ $=$ $\Phi \left(
l_{h_{2}}^{q_{2}}\right) \Phi \left( l_{h_{1}}^{q_{1}}\right) ,$
\end{center}

for all $h_{1},$ $h_{2}$ $\in $ $H,$ and $q_{1},$ $q_{2}$ $\in $ $\{1,$ $%
*\}. $ Therefore, by (4.2.8), we can have that

(4.2.9)

\begin{center}
$\Phi \left( T_{1}T_{2}\right) $ $=$ $\Phi (T_{2})$ $\Phi (T_{1}),$ in $%
Toep^{R}(H),$
\end{center}

for all $T_{1},$ $T_{1}$ $\in $ $Toep(H).$

To show this morphism $\Phi $ is an anti-$*$-isomorphic, it suffices to show
that $\Phi $ preserves the relation (4.2.3) in $Toep(H)$ to the relation
(4.2.6) in $Toep^{R}(H)$:

$\qquad \qquad <h_{1},$ $h_{2}$ $>_{H}$ $1_{L}$ $=$ $\Phi \left(
l_{h_{1}}^{*}l_{h_{2}}\right) $

$\qquad \qquad \qquad \qquad \qquad =$ $\Phi \left( l_{h_{2}}\right) $ $\Phi
\left( l_{h_{1}}^{*}\right) $

by (4.2.8) and (4.2.9)

$\qquad \qquad \qquad \qquad \qquad =$ $r_{h_{2}}^{*}$ $r_{h_{1}}$ $=$ $%
<h_{1},$ $h_{2}>_{H}$ $1_{R}.$

This shows that $\Phi $ is a bijective anti-multiplicative isometric linear
transformation preserving (4.2.3) to (4.2.6), and hence it is an anti-$*$%
-isomorphism from $Toep(H)$ onto $Toep^{R}(H).$
\end{proof}

The above theorem shows that the (left)\strut Toeplitz algebra $Toep(H)$ and
the right Toeplitz algebra $Toep^{R}(H)$ are anti-$*$-isomorphic. We will
use this results in Section 5.3, later.

\section{Representations of $N$-Tree Operators on Vertex Spaces}

As in Section 4.1, we restrict our interests to the case where given graphs
are $N$-regular trees $\mathcal{T}_{N},$ for $N$ $\in $ $\Bbb{N}.$
Throughout this section, we will use the same notations we used in Section
4.1. We want to represent graph operators $T$ on the vertex space $H_{N}.$
Of course, the vertex spaces $H_{V}$ would be different whenever $N$ varies.
For emphasizing we are working on $N$-regular trees, we call the graph
operators of the graph von Neumann algebra $M_{N},$ the $N$-\emph{tree
operators} (or \emph{tree operators}).

In the first two following subsections, we consider the special cases where $%
N$ $=$ $1,$ and $N$\thinspace $=$ $2,$ respectively. And then in Subsection
4.2.3, we will consider the general case.

\subsection{$1$-Tree Operators}

In this subsection, we consider the $1$-regular tree $\mathcal{T}_{1},$ and
its corresponding Hilbert space $H_{1},$ and von Neumann algebra $M_{1}.$ We
will represent the graph operators of $M_{1}$ on the vertex space $H_{V}$ of 
$H_{1}.$

When $N$ $=$ $1,$ the vertex space $H_{V}$\strut $=$ $l^{2}\left( V(\mathcal{%
T}_{N})\right) $ of the graph Hilbert space $H_{1}$ of $\mathcal{T}_{1}$ is
Hilbert-space isomorphic to the $l^{2}$-space $l^{2}(\Bbb{N})$ $\overset{%
\text{Hilbert}}{=}$ $\Bbb{C}^{\oplus \infty }.$ i.e.,

(5.1.1)

\begin{center}
$H_{V}$ $\overset{\text{Hilbert}}{=}$ $l^{2}(\Bbb{N}),$
\end{center}

if $N$ $=$ $1.$ Thus, we will use the isomorphic Hilbert spaces $H_{V}$ and $%
l^{2}(\Bbb{N}),$ alternatively.

Now, put the name (or indices) of vertices of $\mathcal{T}_{1}$ by $\Bbb{N}.$
i.e.,

\strut

\begin{center}
$\mathcal{T}_{1}$ $=$ \quad $\underset{1}{\bullet }\longrightarrow \underset{%
2}{\bullet }\longrightarrow \underset{3}{\bullet }\longrightarrow \underset{4%
}{\bullet }\longrightarrow \cdot \cdot \cdot .$
\end{center}

\strut

Then all reduced finite paths $w$ of the graph groupoid $\Bbb{G}_{1}$ are
expressed by

\begin{center}
$(j,$ $j+1,$ $j+2,$ ..., $j+k)$
\end{center}

or

\begin{center}
$(j+k,$ $...,$ $j+2,$ $j+1,$ $j).$
\end{center}

Indeed, we can obtain that

(5.1.2)

\begin{center}
$FP_{r}\left( \widehat{\mathcal{T}_{1}}\right) $ $=$ $FP\left( \mathcal{T}%
_{1}\right) $ $\sqcup $ $FP\left( \mathcal{T}_{1}^{-1}\right) .$
\end{center}

Remark here that, in general,

\begin{center}
$FP_{r}(\widehat{G})$ $\supseteq $ $FP(G)$ $\cup $ $FP(G^{-1}),$
\end{center}

for an arbitrary graph $G.$

Now, let's denote the inner product of $H_{V}$ by\thinspace $<,>_{2},$ since 
$H_{V}$ is isomorphic to $l^{2}(\Bbb{N}).$ Now, to represent the $1$-tree
operators $T$ of $M_{1},$ we can use the Fourier expansion with respect to
the inner product $<,>.$ i.e., if $T$ $=$ $\underset{w\in Supp(T)}{\sum }$ $%
t_{w}$ $L_{w},$ then the representation $\alpha _{T}$ of $T$ on $H_{V}$ is

(5.1.3)

\begin{center}
$\alpha _{T}$ $=$ $\sum_{k=1}^{\infty }\sum_{l=1}^{\infty }$ $<T\xi _{k},$ $%
\xi _{l}>_{2}$ $\alpha _{k,l},$
\end{center}

where $A_{k,l}$ are the rank-one operators on $l^{2}(\Bbb{N})$,

\begin{center}
$\alpha _{k,l}$ $=$ $\mid l><k\mid ,$ for all $k,$ $l$ $\in $ $\Bbb{N}.$
\end{center}

\strut and

\begin{center}
$\xi _{k}$ $=$ $\left( \underset{(k-1)\text{-times}}{\underbrace{%
0,............,0}},\text{ }1,\text{ }0,\text{ }0,\text{ ...}\right) $ $\in $ 
$l^{2}(\Bbb{N})$ $=$ $H_{V},$
\end{center}

for all $k$ $\in $ $\Bbb{N}.$

Here, $\mid \cdot ><\cdot \mid $ means the Dirac-operator notation. Here,
notice that the inner product $<,>_{2}$ in (5.1.3) means the inner product
on $H_{V},$ not the inner product on the graph Hilbert space $H_{1}.$

By (5.1.3), we can define an action $\alpha $ of the graph von Neumann
algebra $M_{1},$ acting on the vertex space $H_{V}$, satisfying that

(5.1.4)

\begin{center}
$\alpha (T)$ $\overset{def}{=}$ $\alpha _{T},$ for all $T$ $\in $ $M_{N}.$
\end{center}

Then this morphism $\alpha $ is indeed a well-defined action of $M_{N},$
since it is bounded linear, and

\begin{center}
$\alpha (T_{1}T_{2})$ $=$ $\alpha _{T_{1}T_{2}}$ $=$ $\alpha _{T_{1}}$ $%
\alpha _{T_{2}}$ $=$ $\alpha (T_{1})\circ \alpha (T_{2}),$
\end{center}

and

\begin{center}
$\alpha (T_{1}^{*})$ $=$ $\alpha _{T_{1}^{*}}$ $=$ $\left( \alpha
(T_{1})\right) ^{*},$
\end{center}

for all $T_{1},$ $T_{2}$ $\in $ $M_{N},$ where ($\circ $) means the usual
composition.

Therefore, we can obtain the following lemma.

\begin{lemma}
Let $e$ $=$ $(j,$ $j+1)$ be an edge of the $1$-regular tree $\mathcal{T}_{1}$
of (5.1.2), with its shadow $e^{-1}$ $=$ $(j$ $+$ $1,$ $j),$ for $j$ $\in $ $%
\Bbb{N}.$ Then the graph operator $L_{e}$ of $M_{1}$ is unitarily equivalent
to the operator $\alpha _{L_{e}}$ $\in $ $B\left( H_{V}\right) $, where

(5.1.5)

\begin{center}
$\alpha _{L_{e}}$ $\overset{\text{U.E}}{=}$ $\left( 
\begin{array}{lllllll}
0 & 0 &  &  &  &  & 0 \\ 
& \ddots & \ddots &  &  &  &  \\ 
&  & 0 & 0 &  &  &  \\ 
&  &  & \frame{$1$} & 1 &  &  \\ 
&  &  &  & 0 & 0 &  \\ 
&  &  &  &  & \ddots & \ddots \\ 
0 &  &  &  &  &  & \ddots
\end{array}
\right) ,$
\end{center}

with

\begin{center}
$\alpha _{L_{e}}^{*}$ $=$ $\alpha _{L_{e}^{*}}$ $=$ $\alpha _{L_{e^{-1}}}$ $%
\overset{\text{U.E}}{=}$ $\left( 
\begin{array}{lllllll}
0 &  &  &  &  &  & 0 \\ 
0 & \ddots &  &  &  &  &  \\ 
& \ddots & 0 &  &  &  &  \\ 
&  & 0 & \frame{$1$} &  &  &  \\ 
&  &  & 1 & 0 &  &  \\ 
&  &  &  & 0 & \ddots &  \\ 
0 &  &  &  &  & \ddots & \ddots
\end{array}
\right) ,$
\end{center}

where \frame{1} means the $(j,$ $j)$-th position, for $j$ $\in $ $\Bbb{N}$,
and where $\overset{\text{U.E}}{=}$ means ``being unitarily equivalent.'' $%
\square $
\end{lemma}

\strut The proof is straightforward by (5.1.4), and (5.1.3).

Consider now the operator $T_{E}$,

\begin{center}
$T_{E}$ $=$ $\underset{e\in E(\mathcal{T}_{1})}{\sum }$ $L_{e}$
\end{center}

in $M_{1}.$ i.e., this operator $T_{E}$ is the infinite sum of the graph
operators $L_{e}$'s, for all $e$ $\in $ $E(\mathcal{T}_{1}).$ Then it is
represented on $l^{2}(\Bbb{N})$ $=$ $H_{V}$ by the operator, unitarily
equivalent to

\begin{center}
$\left( 
\begin{array}{llllll}
1 & 1 & 0 & \cdots & \cdots &  \\ 
0 & 1 & 1 & \ddots &  &  \\ 
\,\vdots & \ddots & \ddots & \ddots & \ddots &  \\ 
&  & \ddots & 1 & 1 &  \\ 
&  &  &  & \ddots & \ddots \\ 
&  &  &  &  & 
\end{array}
\right) .$
\end{center}

Remark that the identity operator $1_{M_{1}}$ $=$ $\underset{v\in V(\mathcal{%
T}_{1})}{\sum }$ $L_{v}$ of $M_{1}$ is unitarily equivalent to the diagonal
infinite matrix

\begin{center}
$\left( 
\begin{array}{lllll}
1 &  &  &  & 0 \\ 
& 1 &  &  &  \\ 
&  & 1 &  &  \\ 
&  &  & \ddots &  \\ 
0 &  &  &  & \ddots
\end{array}
\right) $
\end{center}

on $H_{V}.$ Thus, we can easily check that

(5.1.6)

\begin{center}
$T_{E}$ $-$ $1_{M_{1}}$ $\overset{\text{U.E}}{=}$ $\left( 
\begin{array}{llllll}
0 & 1 &  &  &  & 0 \\ 
& 0 & 1 &  &  &  \\ 
&  & \ddots & \ddots &  &  \\ 
&  &  & \ddots & 1 &  \\ 
&  &  &  & 0 & \ddots \\ 
0 &  &  &  &  & \ddots
\end{array}
\right) ,$
\end{center}

on $H_{V},$ and it is unitarily equivalent to the adjoint $U^{*}$ of the
unilateral shift $U.$

Recall that the unilateral shift $U$ on $l^{2}(\Bbb{N})$ is the operator
defined by

\begin{center}
$U$ $:$ $(t_{1},$ $t_{2},$ $t_{3},$ ...$)$ $\longmapsto $ $(t_{2},$ $t_{3},$ 
$...),$
\end{center}

for all $(t_{n})_{n=1}^{\infty }$ $\in $ $l^{2}(\Bbb{N}).$ So, the adjoint $%
U^{*}$ of $U$ is the operator, satisfying

\begin{center}
$U^{*}$ $:$ $(t_{1},$ $t_{2},$ $t_{3},$ ...$)$ $\longmapsto $ $(0,$ $t_{1},$ 
$t_{2},$ ...$),$
\end{center}

on $l^{2}(\Bbb{N}).$

Recall also that the \emph{classical Toeplitz algebra} $\mathcal{U}_{1}$ is
the $C^{*}$-subalgebra $C^{*}(U)$ of $B\left( l^{2}(\Bbb{N})\right) ,$
generated by the unilateral shift $U.$ Notice that the Toeplitz algebra $%
\mathcal{U}_{1}$ is also understood as the $C^{*}$-subalgebra of $B\left(
H^{2}(\Bbb{T})\right) ,$ generated by the classical Toeplitz operators $%
T_{\varphi },$ with their symbols $\varphi ,$ contained in the von Neumann
algebra $L^{\infty }(\Bbb{T}),$ where $\Bbb{T}$ is the unit circle in $\Bbb{C%
}$. Here, the Hilbert space $H^{2}(\Bbb{T}),$ where $T_{\varphi }$'s acting
on, is the \emph{Hardy space} consisting of all analytic functions on $\Bbb{T%
},$ which is a subspace of the $L^{2}$-space $L^{2}(\Bbb{T})$ equipped with
the\emph{\ Haar measure}.

From the above observation, we can conclude that:

\begin{theorem}
Let $\mathcal{U}_{1}$ be the classical Toeplitz algebra. Then $\mathcal{U}%
_{1}$ is a $C^{*}$-subalgebra of the graph von Neumann algebra $M_{1}.$
\end{theorem}

\begin{proof}
Let $\mathcal{T}_{1}$ be the $1$-regular tree and $M_{1},$ the corresponding
graph von Neumann algebra of $\mathcal{T}_{1}$ (generated by all 1-tree
operators). And let $\mathcal{U}_{1}$ be the classical Toeplitz algebra $%
C^{*}(U)$ acting on $l^{2}(\Bbb{N}),$ where $U$ is the unilateral shift. By
(5.1.5), and (5.1.6), the unilateral shift $U$ is unitarily equivalent to

\begin{center}
$U$ $\overset{\text{U.E}}{=}$ $T_{E}$ $-$ $1_{M_{1}},$
\end{center}

where

\begin{center}
$T_{E}$ $=$ $\underset{e\in E(\mathcal{T}_{1})}{\sum }$ $L_{e}$ $\in $ $%
M_{1},$
\end{center}

and

\begin{center}
$1_{M_{1}}$ $=$ $\underset{v\in V(\mathcal{T}_{1})}{\sum }$ $L_{v}$ $\in $ $%
M_{1}.$
\end{center}

Therefore, the classical Toeplitz algebra $\mathcal{U}_{1}$ satisfies that

\begin{center}
$\mathcal{U}_{1}$ $\overset{def}{=}$ $C^{*}(U)$ $\overset{*\text{-iso}}{=}$ $%
C^{*}\left( \alpha (T_{E}-1_{M_{1}})\right) $ $\overset{*\text{-iso}}{=}$ $%
C^{*}\left( \alpha (T_{E})\right) ,$
\end{center}

in $B\left( l^{2}(\Bbb{N})\right) $ $\overset{*\text{-iso}}{=}$ $B\left(
H_{V}\right) ,$ where $\alpha $ is the action of $M_{1}$ acting on $l^{2}(%
\Bbb{N}),$ in the sense of (5.1.4). Therefore, by the very definition of $%
M_{1},$

\begin{center}
$M_{1}$ $=$ $vN\left( L(\Bbb{G}_{1})\right) $ in $B(H_{1}),$
\end{center}

the algebra $\mathcal{U}_{1}$ is a $C^{*}$-subalgebra of $M_{1}.$
\end{proof}

\strut The above theorem shows that all classical Toeplitz operators are the
representations of certain elements of the graph von Neumann algebra $M_{1}$ 
$=$ $M_{\mathcal{T}_{1}}.$ In particular, the generator $U$ of $\mathcal{U}%
_{1}$ is the infinite sum of graph operators, by (5.1.5).

Notice now that all (classical) Toeplitz operators can be understood as the
certain infinite sum of certain graph operators with pattern.

Let $w$ $=$ $(j,$ $j$ $+$ $1,$ $j$ $+$ $2)$ be a length-2 reduced finite
path in the graph groupoid $\Bbb{G}_{1}$ of $\mathcal{T}_{1}.$ Then, by
(5.1.3), the graph operator $L_{w}$ $\in $ $M_{1}$, induced by $w,$ is
unitarily equivalent to the infinite matrix $\alpha _{L_{w}}$ $=$ $\alpha
(L_{w}),$

\begin{center}
$L_{w}$ $\overset{\text{U.E}}{=}$ $\alpha _{L_{w}}$ $\overset{\text{U.E}}{=}$
$\left( 
\begin{array}{lllllll}
0 & 0 & 0 &  &  &  & 0 \\ 
& \ddots & \ddots & \ddots &  &  &  \\ 
&  & 0 & 0 & 0 &  &  \\ 
&  &  & \frame{$1$} & 0 & 1 &  \\ 
&  &  &  & 0 & 0 & 0_{\ddots } \\ 
&  &  &  &  & \ddots & \ddots \\ 
0 &  &  &  &  &  & \ddots
\end{array}
\right) ,$
\end{center}

\strut

on $l^{2}(\Bbb{N})$ $=$ $H_{V},$ satisfying

\strut

\begin{center}
$
\begin{array}{ll}
L_{w}^{*}=L_{w^{-1}} & \overset{\text{U.E}}{=}\alpha _{L_{w^{-1}}}=\alpha
_{L_{w}}^{*} \\ 
& \overset{\text{U.E}}{=}\left( 
\begin{array}{lllllll}
0 &  &  &  &  &  & 0 \\ 
0 & \ddots &  &  &  &  &  \\ 
0 & \ddots & 0 &  &  &  &  \\ 
& \ddots & 0 & \frame{$1$} &  &  &  \\ 
&  & 0 & 0 & 0 &  &  \\ 
&  &  & 1 & 0 & \ddots &  \\ 
0 &  &  &  & 0_{\ddots } & \ddots & \ddots
\end{array}
\right) ,
\end{array}
$
\end{center}

\strut

on $l^{2}(\Bbb{N}),$ where \frame{$1$} means the $(j$, $j)$-th entry, for
all $j$ $\in $ $\Bbb{N}.$

So, inductively, we obtain that:

\begin{proposition}
Let $w$ $=$ $(j,$ $j$ $+$ $1,$ ..., $j$ $+$ $k)$ $\in $ $\Bbb{G}_{1},$ for $%
j $, $k$ $\in $ $\Bbb{N}.$ Then the corresponding graph operator $L_{w}$ of $%
M_{1}$ is unitarily equivalent to the operator $\alpha _{L_{w}}$,

(5.1.7)

\begin{center}
$\left( 
\begin{array}{lllllll}
0 & \cdot \cdot \cdot \cdot \cdot \cdot \cdot \,0 &  &  &  &  & 0 \\ 
& \ddots &  & \ddots &  &  &  \\ 
&  & 0 & \cdot \cdot \cdot \cdot \cdot \cdot \cdot \cdot & 
\,\,\,\,\,\,\,\,\,\,\,\,0 &  &  \\ 
&  &  & \,\,\,\,\,\,\,\,\,\,\,\,\frame{$1$} & \overset{k\text{-steps}}{%
\overbrace{\cdot \cdot \cdot \cdot \cdot \cdot \cdot \cdot }} & 
\,\,\,\,\,\,\,\,\,\,1 &  \\ 
&  &  &  & \,\,\,\,\,\,\,\,\,\,\,\,0 & \cdot \cdot \cdot \cdot \cdot \cdot
\cdot \cdot & 0_{\ddots } \\ 
&  &  &  &  & \ddots &  \\ 
0 &  &  &  &  &  & \ddots
\end{array}
\right) ,$
\end{center}

on $l^{2}(\Bbb{N})$ $=$ $H_{V}.$ i.e., $\alpha _{L_{w}}$ is represented as
an infinite matrix with only nonzero $(j,$ $j)$ and $(j$ $+$ $k,$ $j)$
entries $1$'s. $\square $
\end{proposition}

The proof of the above proposition is straightforward by (5.1.3). Now,
Define an element $T_{E(k)}$ in $M_{1}$ by

\begin{center}
$T_{(k)}$ $\overset{def}{=}$ $\underset{w\in E(+k)}{\sum }$ $L_{w},$
\end{center}

where the support $E(+k)$ of $T_{(k)}$ is defined by the subset

(5.1.8)

\begin{center}
$E(+k)$ $\overset{def}{=}$ $\{w$ $\in $ $FP(\mathcal{T}_{1})$ $:$ $\left|
w\right| $ $=$ $k\},$
\end{center}

consisting of all length-$k$ (non-reduced) finite paths of the finite path
set $FP(\mathcal{T}_{1})$ of $\mathcal{T}_{1},$ for all $k$ $\in $ $\Bbb{N}.$
Here, remark again that the ``reduced'' finite path set $FP_{r}\left( 
\widehat{\mathcal{T}_{1}}\right) $ of the shadowed graph $\widehat{\mathcal{T%
}_{1}}$ of $\mathcal{T}_{1}$ is identified with the disjoint union of the
``non-reduced'' finite path set $FP\left( \mathcal{T}_{1}\right) $ of $%
\mathcal{T}_{1}$ and the non-reduced finite path set $FP\left( \mathcal{T}%
_{1}^{-1}\right) $ of the shadow $\mathcal{T}_{1}^{-1}$ of $\mathcal{T}_{1}$
(See (5.1.2)).

By (5.1.8), we have that

(5.1.9)

\begin{center}
$FP\left( \mathcal{T}_{1}\right) $ $=$ $\underset{k=1}{\overset{\infty }{%
\sqcup }}$ $E(+k).$
\end{center}

Also, by (5.1.2) and (5.1.8), we can obtain the subsets $E(-k)$ of the
reduced finite path set $FP_{r}\left( \widehat{\mathcal{T}_{1}}\right) ,$
where

(5.1.10)

\begin{center}
$E(-k)$ $\overset{def}{=}$ $\{w$ $\in $ $FP(\mathcal{T}_{1}^{-1})$ $:$ $%
\left| w\right| $ $=$ $k\},$
\end{center}

for all $k$ $\in $ $\Bbb{N}$, satisfying that

(5.1.11)

\begin{center}
$FP\left( \mathcal{T}_{1}^{-1}\right) $ $=$ $\underset{k=1}{\overset{\infty 
}{\sqcup }}$ $E(-k),$
\end{center}

and hence

(5.1.12)

\begin{center}
$FP\left( \widehat{\mathcal{T}_{1}}\right) $ $=$ $\left( \underset{k=1}{%
\overset{\infty }{\sqcup }}\text{ }E(+k)\right) $ $\cup $ $\left( \underset{%
k=1}{\overset{\infty }{\sqcup }}\text{ }E(-k)\right) ,$
\end{center}

by (5.1.2) and (5.1.11). Thanks to (5.1.10) and (5.1.12), we can define an
element $T_{(-k)}$ of $M_{1}$ by

\begin{center}
$T_{(-k)}$ $\overset{def}{=}$ $\underset{w\in E(-k)}{\sum }$ $L_{w}.$
\end{center}

Then we can easily check

(5.1.13)

\begin{center}
$T_{(+k)}^{*}$ $=$ $T_{(-k)},$ for all $k$ $\in $ $\Bbb{N},$
\end{center}

since $L_{w}^{*}$ $=$ $L_{w^{-1}},$ for all $w$ $\in $ $\Bbb{G}_{1}.$ Also,
the operator $T_{E}$, defined at the beginning of this section, is nothing
but the element $T_{(+1)},$ in $M_{1}.$

Therefore, by the above discussion, we can obtain the following lemma.

\begin{lemma}
For $k$ $\in $ $\Bbb{N},$ let

\begin{center}
$T_{(+k)}$ $\overset{def}{=}$ $\underset{w\in E(+k)}{\sum }$ $L_{w},$ and $%
T_{(-k)}$ $\overset{def}{=}$ $\underset{x\in E(-k)}{\sum }$ $L_{x},$
\end{center}

in $M_{1},$ where $E(+k),$ and $E(-k)$ are defined in (5.1.8) and (5.1.10),
respectively. Then

\begin{center}
$T_{(+k)}^{*}$ $=$ $T_{(-k)}$ in $M_{1},$ for all $k$ $\in $ $\Bbb{N},$
\end{center}

and

(5.1.14)

$\qquad \quad T_{(+k)}$ $\overset{\text{U.E}}{=}$ $\alpha _{T_{(+k)}}$ $=$ $%
\alpha _{(+k)}$ $\overset{\text{U.E}}{=}$

\begin{center}
$\left( 
\begin{array}{llllll}
1 & \overset{k\text{-times}}{\overbrace{0\cdot \cdot \cdot \cdot \cdot 0}} & 
1 &  &  & 0 \\ 
& \,\,\,\,\,\,\,\,\,\,\,\,1 & 0\cdot \cdot \cdot \cdot \cdot 0 & 1 &  &  \\ 
&  & \,\,\,\,\,\,\,\,\,\,\,\,1 & 0\cdot \cdot \cdot \cdot \cdot 0 & 1 &  \\ 
&  &  & \,\,\,\,\,\,\,\,\ddots & \ddots & \ddots \\ 
&  &  &  &  &  \\ 
0 &  &  &  &  & 
\end{array}
\right) ,$
\end{center}

on $l^{2}(\Bbb{N}),$ for all $k$ $\in $ $\Bbb{N}.$ Hence, the element $%
T_{(-k)}$ is unitarily equivalent to the operator $\alpha _{(+k)}^{*}$ on $%
l^{2}(\Bbb{N}).$ i.e.,

(5.1.15)

\begin{center}
$T_{(-k)}$ $\overset{\text{U.E}}{=}$ $\alpha _{T_{(-k)}}$ $\overset{denote}{=%
}$ $\alpha _{(-k)}$ $=$ $\alpha _{(+k)}^{*}.$
\end{center}

$\square $
\end{lemma}

By the above lemma, we can obtain the following proposition.

\begin{proposition}
Let $T_{(+k)}$ and $T_{(-k)}$ be given as in the above lemma in the graph
von Neumann algebra $M_{1}$ of the $1$-regular tree $\mathcal{T}_{1},$ for $%
k $ $\in $ $\Bbb{N}.$ Then the elements $T_{(+k)}$ $-$ $1_{M_{1}},$ and $%
T_{(-k)}$ $-$ $1_{M_{1}}$ are unitarily equivalent to $U^{*k},$ and $U^{k}$
on $l^{2}(\Bbb{N})$ $=$ $H_{V},$ respectively for all $k$ $\in $ $\Bbb{N}$,
where $U$ is the unilateral shift on $l^{2}(\Bbb{N}).$
\end{proposition}

\begin{proof}
By \strut (5.1.14) and (5.1.15), we can have that

(5.1.16)

\begin{center}
$
\begin{array}{ll}
T_{(+k)}-1_{M_{1}} & \overset{\text{U.E}}{=}\alpha _{(+k)}-I \\ 
& \overset{\text{U.E}}{=}\left( 
\begin{array}{lllllllll}
0 & \cdots & 0 & \frame{$1$} &  &  &  &  & 0 \\ 
& \ddots & \cdots & \ddots & \ddots &  &  &  &  \\ 
&  & 0 & \cdots & 0 & 1 &  &  &  \\ 
&  &  & 0 & \cdots & 0 & 1 &  &  \\ 
&  &  &  & 0 & \cdots & 0 & 1 &  \\ 
0 &  &  &  &  & \ddots & \cdots & \ddots & \ddots
\end{array}
\right) ,
\end{array}
$
\end{center}

on $l^{2}(\Bbb{N})$ $=$ $H_{V},$ where \frame{1} is the $(k,$ $1)$-entry,
for all $k$ $\in $ $\Bbb{N},$ and where $I$ means the identity operator on $%
l^{2}(\Bbb{N}).$ By (5.1.16), we obtain that

(5.1.17)

\begin{center}
$
\begin{array}{ll}
T_{(-k)}-1_{M_{1}} & \overset{\text{U.E}}{=}\alpha _{(-k)}-I \\ 
& \overset{\text{U.E}}{=}\left( 
\begin{array}{llllll}
0 &  &  &  &  & 0 \\ 
\,\vdots & \ddots &  &  &  &  \\ 
0 & \,\,\vdots & 0 &  &  &  \\ 
\frame{$1$} & \ddots & \,\vdots & \ddots &  &  \\ 
& \ddots & 0 &  &  &  \\ 
&  & 1 & \ddots &  &  \\ 
&  &  & \ddots &  &  \\ 
0 &  &  &  &  & 
\end{array}
\right) ,
\end{array}
$
\end{center}

on $l^{2}(\Bbb{N}),$ where \frame{1} is the $(1,$ $k)$-th entry. Therefore,
the element $T_{(+k)}$ $-$ $1_{M_{1}}$ (resp., $T_{(-k)}$ $-$ $1_{M_{1}}$)
is unitarily equivalent to $U^{*k}$ (resp., $U^{k}$), for all $k$ $\in $ $%
\Bbb{N},$ by (5.1.16) (resp., by (5.1.17)), on $l^{2}(\Bbb{N}),$ for all $k$ 
$\in $ $\Bbb{N}.$
\end{proof}

\strut The above proposition is the generalization of (5.1.5). Under the
settings of the above proposition, we can re-write (5.1.5) that: the element 
$T_{(+1)}$ $-$ $1_{M_{1}}$ of $M_{1}$ is unitarily equivalent to the adjoint
of the unilateral shift $U^{*}$ on $l^{2}(\Bbb{N})$; and the element $%
T_{(-1)}$ $-$ $1_{M_{1}}$ of $M_{1}$ is unitarily equivalent to the
unilateral shift $U$ on $l^{2}(\Bbb{N}).$

By the above proposition, we can obtain the following theorem.

\begin{theorem}
Let $S$ $=$ $\sum_{j=-n}^{-1}$ $t_{j}$ $U^{*\,j}$ $+$ $t_{0}$ $I$ $+$ $%
\sum_{i=1}^{k}$ $t_{i}U^{i}$ be a Toeplitz operator, with $t_{p}$ $\in $ $%
\Bbb{C},$ for $p$ $=$ $-n,$ ..., $-1,$ $0,$ $1,$ ..., $k,$ in the classical
Toeplitz algebra $\mathcal{U}_{1}$ (or equivalently, $S$ is the Toeplitz
operator $T_{\varphi }$ with its trigonometric polynomial symbol $\varphi
(z) $ $=$ $\sum_{p=-n}^{k}$ $t_{p}$ $z^{p}$ in $L^{\infty }(\Bbb{T}),$ on $%
H^{2}(\Bbb{T})$). Let $T_{(\pm k)}$ be the elements of the graph von Neumann
algebra $M_{1}$ of the 1-regular tree $\mathcal{T}_{1}$, for all $k$ $\in $ $%
\Bbb{N}.$ Then the Toeplitz operator $S$ is unitarily equivalent to the
element $S^{\prime }$ of $M_{1},$ on $l^{2}(\Bbb{N})$ $=$ $H_{V},$ where

(5.1.18)

\begin{center}
$S^{\prime }$ $=$ $\sum_{j=1}^{n}$ $t_{-j}$ $T_{(+j)}$ $+$ $s_{0}$ $%
1_{M_{1}} $ $+$ $\sum_{i=1}^{k}$ $t_{i}$ $T_{(-i)},$
\end{center}

with

\begin{center}
$s_{0}$ $=$ $t_{0}$ $-$ $\left( \sum_{j=-n}^{-1}\text{ }t_{j}\right) $ $-$ $%
\left( \sum_{i=1}^{k}t_{i}\right) $ in $\Bbb{C}.$
\end{center}
\end{theorem}

\begin{proof}
\strut By the above proposition, the operators $U^{k}$ and $U^{*\,l}$ of $%
B\left( l^{2}(\Bbb{N})\right) $ are unitarily equivalent to the elements $%
T_{(-k)}$ $-$ $1_{M_{1}},$ and $T_{(+k)}$ $-$ $1_{M_{1}}$ of $M_{1}$ on $%
l^{2}(\Bbb{N}),$ respectively, where $U$ is the unilateral shift on $l^{2}(%
\Bbb{N}).$ So, the given Toeplitz operator $S$ is unitarily equivalent to
the element,

$\qquad S^{\prime }$ $=$ $\sum_{j=-n}^{-1}t_{j}\left(
T_{(-j)}-1_{M_{1}}\right) $ $+$ $t_{0}$ $1_{M_{1}}$ $+$ $\sum_{i=1}^{k}$ $%
t_{i}\left( T_{(-i)}-1_{M_{1}}\right) $

$\qquad \qquad =$ $\left(
\sum_{j=-n}^{-1}t_{j}T_{(-j)}-\sum_{j=1}^{-1}t_{j}1_{M_{1}}\right) $ $+$ $%
t_{0}$ $1_{M_{1}}$

$\qquad \qquad \qquad \qquad +$ $\left(
\sum_{i=-n}^{k}t_{i}T_{(-i)}-\sum_{i=1}^{k}t_{i}1_{M_{1}}\right) $

$\qquad \qquad =$ $\left( \sum_{j=-n}^{-1}\text{ }t_{j}\text{ }%
T_{(-j)}\right) $ $+$ $\left( \sum_{i=1}^{k}t_{i}\text{ }T_{(-i)}\right) $

$\qquad \qquad \qquad \qquad $ $-$ $\left( \sum_{j=-n}^{-1}\text{ }%
t_{j}\right) $ $1_{M_{1}}$ $+$ $t_{0}$ $1_{M_{1}}$ $-$ $\left(
\sum_{i=1}^{k}t_{i}\right) $ $1_{M_{1}}$

$\qquad \qquad =$ $\left( \sum_{j=-n}^{-1}\text{ }t_{j}\text{ }%
T_{(-j)}\right) $ $+$ $\left( \sum_{i=1}^{k}t_{i}\text{ }T_{(-i)}\right) $

$\qquad \qquad \qquad \qquad +$ $\left( t_{0}-\left(
\sum_{j=-n}^{-1}t_{j}\right) -\left( \sum_{i=1}^{k}t_{i}\right) \right) $ $%
1_{M_{1}}$

\strut
\end{proof}

\strut The above theorem characterizes the classical Toeplitz operators in
terms of 1-tree operators.

\begin{corollary}
The classical Toeplitz algebra $\mathcal{U}_{1}$ is $*$-isomorphic to the $%
C^{*}$-subalgebra

\begin{center}
$C^{*}\left( \alpha \{T_{(n)}\in M_{1}:n\in \Bbb{Z},\text{ with }T_{(0)}%
\overset{def}{=}1_{M_{1}}\}\right) $
\end{center}

of $B\left( l^{2}(\Bbb{N})\right) ,$ where $\alpha $ is the action of $M_{1}$%
, in the sense of (5.1.4), acting on $H_{V}$ $=$ $l^{2}(\Bbb{N}).$ $\square $
\end{corollary}

\strut \strut The above corollary also shows that the classical Toeplitz
algebra $\mathcal{U}_{1}$ is indeed a $C^{*}$-subalgebra of the graph von
Neumann algebra $M_{1}$ of the 1-regular tree $\mathcal{T}_{1}.$

\subsection{$2$-Tree Operators}

In this section, we will consider the relation between the generalized
Toeplitz operators in $Toep(\Bbb{C}^{\oplus 2})$ and $2$-tree operators
which are the graph operators in the graph von Neumann algebra $M_{2}$ of
the $2$-regular tree $\mathcal{T}_{2},$

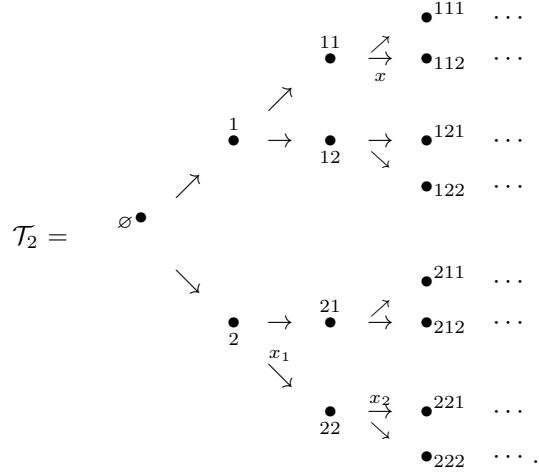
\begin{center}
$\mathcal{T}_{2}$ $=$ \quad $
\begin{array}{lllllllll}
&  &  &  &  &  & \bullet ^{111} & \cdots &  \\ 
&  &  &  & \overset{11}{\bullet } & \overset{\nearrow }{\underset{x}{%
\rightarrow }} & \bullet _{112} & \cdots &  \\ 
&  &  & \nearrow &  &  &  &  &  \\ 
&  & \overset{1}{\bullet } & \rightarrow & \underset{12}{\bullet } & 
\underset{\searrow }{\rightarrow } & \bullet ^{121} & \cdots &  \\ 
& \nearrow &  &  &  &  & \bullet _{122} & \cdots &  \\ 
_{\varnothing }\bullet &  &  &  &  &  &  &  &  \\ 
&  &  &  &  &  &  &  &  \\ 
& \searrow &  &  &  &  & \bullet ^{211} & \cdots &  \\ 
&  & \underset{2}{\bullet } & \rightarrow & \overset{21}{\bullet } & 
\overset{\nearrow }{\rightarrow } & \bullet _{212} & \cdots &  \\ 
&  &  & \overset{x_{1}}{\searrow } &  &  &  &  &  \\ 
&  &  &  & \underset{22}{\bullet } & \underset{\searrow }{\overset{x_{2}}{%
\rightarrow }} & \bullet ^{221} & \cdots &  \\ 
&  &  &  &  &  & \bullet _{222} & \cdots . & 
\end{array}
$
\end{center}

\strut As we discussed in Section 4.1, the corresponding graph Hilbert space 
$H_{2}$ of $\mathcal{T}_{2}$ is decomposed by the vertex space $H_{V}$ and
its orthogonal complemented subspace $H_{V}^{\perp }$. Like in Section 5.1,
we want to represent graph operators in the graph von Neumann algebra $M_{2}$
as an operator on $H_{V}.$ To do that we first concentrate on characterize $%
H_{V}.$

Now, we denote the process sending the vertex

\begin{center}
$i_{1}i_{2}$ ... $i_{n}$
\end{center}

\strut in the $n$-th level of $\mathcal{T}_{2}$ to the vertex

\begin{center}
$i_{1}i_{2}$ ... $i_{n}$ $1$
\end{center}

in the $(n$ $+$ $1)$-th level of $\mathcal{T}_{2}$ by $\gamma _{1},$ for all 
$n$ $\in $ $\Bbb{N}$, where $i_{1},$ ..., $i_{n}$ $\in $ $\{1,$ $2\}.$

Similarly, we denote the process sending

\begin{center}
$i_{1}$ ... $i_{n}$ to $i_{1}$ ... $i_{n}$\strut $2$
\end{center}

\strut by $\gamma _{2},$ for all $n$ $\in $ $\Bbb{N}.$ i.e., $\gamma _{j}$
are the function on the vertex set $V\left( \mathcal{T}_{2}\right) $ of $%
\mathcal{T}_{2}$, defined by

(5.2.1)

\begin{center}
$\gamma _{j}\left( i_{1}i_{2}\text{ ... }i_{n}\right) $ $\overset{def}{=}$ $%
i_{1}$ $i_{2}$ ... $i_{n}$ $j,$
\end{center}

for all $j$ $=$ $1,$ $2,$ for all $n$ $\in $ $\Bbb{N}.$ Also, we can
understand these functions $\gamma _{j}$ generates the edges in $\mathcal{T}%
_{2}.$ i.e., the function

\begin{center}
$\gamma _{j}\left( i_{1}\text{ }i_{2}\text{ ... }i_{n}\right) $
\end{center}

can be understood as the edge

\begin{center}
$\left( i_{1}i_{2}\text{ ... }i_{n},\text{ }i_{1}i_{2}\text{ ... }%
i_{n}j\right) ,$
\end{center}

in $\mathcal{T}_{2},$ for $j$ $=$ $1,$ $2.$

Now, let $X_{2}$ be the set $\{1,$ $2\},$ and define $\mathcal{X}_{2},$ by
the union of $X_{2}^{*}$ and $\{\emptyset \},$ where $\emptyset $ is the
empty word in $X_{2},$ i.e.,

(5.2.2)

\begin{center}
$\mathcal{X}_{2}$ $=$ $\{\emptyset \}$ $\cup $ $X_{2}^{*},$
\end{center}

where $X_{2}^{*}$ means the set of all words in $X_{2}.$ So, we can
understand the set $\mathcal{X}_{2}$ is the collection of all words in $%
X_{2} $ and the empty word. Then we can regard the functions $\gamma _{j}$
of (5.2.1) on as functions on $\mathcal{X}_{2},$ for all $j$ $=$ $1,$ $2.$

\strut \strut

\textbf{Notation} Let $W$ be a finite words in $X_{2}^{*}.$ Then we denote
the compositions

\begin{center}
$\gamma _{j_{1}}$ $\circ $ $\gamma _{j_{2}}$ $\circ $ ... $\circ $ $\gamma
_{j_{k}}$ on $\mathcal{X}_{2}$
\end{center}

simply by

\begin{center}
$\gamma _{j_{1}j_{2}\text{ ... }j_{k}},$ or $\gamma _{W},$
\end{center}

whenever $W$ $=$ $j_{1}$ $j_{2}$ ... $j_{k}$ $\in $ $X_{2}^{*}.$ $\square $

\strut

The above notation gives us a motivation for the following proposition.

\begin{proposition}
Let $X_{2}^{*}$ be the set of all finite words in $X_{2}$ $=$ $\{1,$ $2\},$
and let $\mathcal{X}_{2}$ be the set defined in (5.2.2). Then there exists
an (right) action $\gamma $ of $X_{2}^{*},$ acting on $\mathcal{X}_{2},$
such that

(5.2.3)

\begin{center}
$\gamma $ $:$ $W$ $\in $ $X_{2}^{*}$ $\longmapsto $ [$\gamma _{W}$ $:$ $%
\mathcal{X}_{2}$ $\rightarrow $ $\mathcal{X}_{2}$],
\end{center}

with additional equality

\begin{center}
$\gamma _{W}(\emptyset )$ $=$ $W,$ for all $W$ $\in $ $\mathcal{X}_{2}.$
\end{center}
\end{proposition}

\begin{proof}
The map $\gamma $ $:$ $X_{2}^{*}$ $\rightarrow $ $\mathcal{F}(\mathcal{X}%
_{2})$ is well-defined, where $\mathcal{F}(\mathcal{X}_{2})$ is the set of
all functions on $\mathcal{X}_{2}.$ So, it is enough to show that

\begin{center}
$\gamma (W_{1}W_{2})$ $=$ $\gamma (W_{2})$ $\circ $ $\gamma (W_{1}),$
\end{center}

on $\mathcal{X}_{2},$ for all $W_{1},$ $W_{2}$ $\in $ $X_{2}^{*}.$ Indeed,
we can have that

\begin{center}
$\gamma (W_{1}W_{2})$ $=$ $\gamma _{W_{1}W_{2}}$ $=$ $\gamma _{W_{2}}$ $%
\gamma _{W_{1}}$ $=$ $\gamma (W_{2})$ $\circ $ $\gamma (W_{1}).$
\end{center}
\end{proof}

The above proposition, indicating the existence of the right action $\gamma $
of $X_{2}^{*}$ on $\mathcal{X}_{2},$ shows that the vertex set $V\left( 
\mathcal{T}_{2}\right) $ is generated by the action $\gamma .$ Therefore, we
can obtain the set-equalities

(5.2.4)

\begin{center}
$V\left( \mathcal{T}_{2}\right) $ $=$ $\mathcal{X}_{2}$ $=$ $\{\emptyset \}$ 
$\cup $ $\left\{ \gamma _{W}(\mathcal{X}_{2}):W\in X_{2}^{*}\right\} $.
\end{center}

\strut Now, consider the vertex space more in detail. Motivated by (5.2.3)
and (5.2.4), we may expect that the vertex space

\begin{center}
$H_{V}$ $=$ $l^{2}(\mathcal{X}_{2})$ $=$ $l^{2}(X_{2}^{*})$ $=$ $l^{2}\left(
V(\mathcal{T}_{2})\right) $
\end{center}

is Hilbert-space isomorphic to \emph{the }(\emph{generalized})\emph{\ Fock
space over} $\Bbb{C}^{\oplus 2},$

(5.2.5)

\begin{center}
$\mathcal{F}_{2}$ $=$ $\mathcal{F}(\Bbb{C}^{\oplus 2}\mathcal{)}$ $=$ $%
\underset{k=0}{\overset{\infty }{\oplus }}$ $\left( \Bbb{C}^{\oplus
2}\right) ^{\otimes \,k},$
\end{center}

with the identity,

\begin{center}
$\left( \Bbb{C}^{\oplus 2}\right) ^{\otimes 0}$ $=$ $\Bbb{C}.$
\end{center}

Let $\mathcal{F}_{H}$ be the Fock space over an arbitrary Hilbert space $H,$
like in Section 4.2. If $H$ is an $n$-dimensional Hilbert space $\Bbb{C}%
^{\oplus n},$ then we denote $\mathcal{F}_{\Bbb{C}^{\oplus n}}$ simply by $%
\mathcal{F}_{n},$ for all $n$ $\in $ $\Bbb{N}.$

\begin{theorem}
Let $H_{V}$ be the vertex space of the graph Hilbert space $H_{2}$ of $%
\mathcal{T}_{2}.$ Then

(5.2.6)

\begin{center}
$H_{V}$ $\overset{\text{Hilbert}}{=}$ $l^{2}\left( \mathcal{X}_{2}\right) $ $%
\overset{\text{Hilbert}}{=}$ $\mathcal{F}_{2}$.
\end{center}
\end{theorem}

\begin{proof}
\strut First, recall the relation (4.1.1), and (4.1.2). In particular, by
(4.1.1), the vertex space $H_{V}$ of the graph Hilbert space $H_{2}$ is
Hilbert-space isomorphic to

(5.2.7)

\begin{center}
$H_{V}$ $\overset{\text{Hilbert}}{=}$ $\underset{v\in V(\mathcal{T}_{2})}{%
\oplus }$ $\Bbb{C}\xi _{v}$ $=$ $\underset{W\in \mathcal{X}_{2}}{\oplus }$ $%
\Bbb{C}\xi _{W},$
\end{center}

where

\begin{center}
$\mathcal{B}_{V}$ $=$ $\{\xi _{v}$ $:$ $v$ $\in $ $V(\mathcal{T}_{2})\}$ $=$ 
$\{\xi _{W}$ $:$ $W$ $\in $ $\mathcal{X}_{2}\}$
\end{center}

is the Hilbert basis (or the orthonormal basis) of $H_{V}.$ Motivated by
(5.2.4), we may determine a linear map, satisfying

(5.2.8)

\begin{center}
$\xi _{j_{1}j_{2}...j_{k}}$ $\in $ $\mathcal{B}_{V}$ $\longmapsto $ $\eta
_{j_{1}}\otimes \eta _{j_{2}}$ $\otimes $ ... $\otimes $ $\eta _{j_{k}}$ $%
\in $ $\mathcal{B}_{k},$
\end{center}

where $\mathcal{B}_{k}$ is the Hilbert basis for $\left( \Bbb{C}^{\oplus
2}\right) ^{\otimes k},$ for all $k$ $\in $ $\Bbb{N},$ where

(5.2.9)

\begin{center}
$\eta _{j_{i}}$ $=$ $\left\{ 
\begin{array}{ll}
(1,\text{ }0) & \text{if }j_{i}=1 \\ 
(0,\text{ }1) & \text{if }j_{i}=2,
\end{array}
\right. $
\end{center}

for all $i$ $=$ $1,$ ..., $k,$ for $k$ $\in $ $\Bbb{N}.$ We denote this
morphism with (5.2.8) by $\Phi _{k},$ for all $k$ $\in $ $\Bbb{N}.$ Then we
define a linear map

(5.2.10)

\begin{center}
$\Phi $ $:$ $\underset{W\in \mathcal{X}_{2}}{\oplus }$ $\Bbb{C}\xi _{W}$ $%
\rightarrow $ $\mathcal{F}_{2}$
\end{center}

by

\begin{center}
$\Phi $ $=$ $\underset{n=0}{\overset{\infty }{\oplus }}$ $\Phi _{n},$ with
the identity $\Phi _{0}$ $=$ $i_{d},$
\end{center}

where $\Phi _{n}$ is the linear map satisfying (5.2.8), for $n$ $\geq $ $1,$
and where $i_{d}$ means the identity map on $\Bbb{C}$ (i.e., $i_{d}(z)$ $=$ $%
z,$ $\forall $ $z$ $\in $ $\Bbb{C}$). Indeed, we can define such a linear
map, because

(5.2.11)

\begin{center}
$\underset{W\in \mathcal{X}_{2}}{\oplus }$ $\Bbb{C}\xi _{W}$ $=$ $\Bbb{C}$ $%
\oplus $ $\left( \underset{n=2}{\overset{\infty }{\oplus }}\left( \underset{%
W\in \mathcal{X}_{2}(k)}{\oplus }\text{ }\Bbb{C}\xi _{W}\right) \right) ,$
\end{center}

where

\begin{center}
$\mathcal{X}_{2}(k)$ $\overset{def}{=}$ $\{W$ $\in $ $\mathcal{X}_{2}$ $:$ $%
\left| W\right| $ $=$ $k\},$ for all $k$ $\in $ $\Bbb{N},$
\end{center}

where $\left| W\right| $ means the length of the word $W.$ So, the linear
map $\Phi $ of (5.2.10) is a well-defined into $\mathcal{F}_{2},$ by
(5.2.9). i.e., the summands

\begin{center}
$\Bbb{C}\xi _{1j_{1}j_{2}\text{ ... }j_{k}}$
\end{center}

of (5.2.11) corresponds to the subspace

\begin{center}
$p_{j_{1}}\left( \Bbb{C}^{\oplus 2}\right) $ $\otimes $ $p_{j_{2}}\left( 
\Bbb{C}^{\oplus 2}\right) $ $\otimes $ ... $\otimes $ $p_{j_{k}}\left( \Bbb{C%
}^{\oplus 2}\right) $
\end{center}

of the summand $\left( \Bbb{C}^{\oplus 2}\right) ^{\otimes k}$ of the Fock
space $\mathcal{F}_{2},$ where $p_{1}$ and $p_{2}$ are the natural
projections on $\Bbb{C}^{\oplus 2},$

\begin{center}
$\left( 
\begin{array}{ll}
1 & 0 \\ 
0 & 0
\end{array}
\right) ,$ and $\left( 
\begin{array}{ll}
0 & 0 \\ 
0 & 1
\end{array}
\right) ,$
\end{center}

respectively.

Therefore, this linear map $\Phi $ is basis-element preserving, and hence it
is bijective. Moreover, it is easy to check that $\left\| \Phi \right\| $ $=$
$1$, by the very definition. So, $\Phi $ is the isometric bijective linear
map, preserving Hilbert bases. Therefore, the Hilbert spaces

\begin{center}
$\underset{W\in \mathcal{X}_{2}}{\oplus }$ $\Bbb{C}\xi _{W}$ and $\mathcal{F}%
_{2}$
\end{center}

are Hilbert-space isomorphic, with its isomorphism $\Phi .$ This shows that

\begin{center}
$H_{V}$ $\overset{\text{Hilbert}}{=}$ $\underset{W\in \mathcal{X}_{2}}{%
\oplus }$ $\Bbb{C}$ $\xi _{W}$ $\overset{\text{Hilbert}}{=}$ $\mathcal{F}%
_{2}.$
\end{center}
\end{proof}

\strut The above theorem characterize the vertex space $H_{V}$ in the graph
Hilbert space $H_{2}$ of $\mathcal{T}_{2}$ by the Fock space $\mathcal{F}%
_{2} $ over $\Bbb{C}^{\oplus 2}.$ So, from now on, we use $H_{V}$ and $%
\mathcal{F}_{2},$ alternatively.

Similar to Section 5.1, we will represent graph operators as operators on
the vertex space $H_{V}$ $=$ $\mathcal{F}_{2}.$

\begin{lemma}
\strut Let $(W,$ $Wj)$ be an edge connecting a vertex

\begin{center}
$W$ $=$ $j_{1}j_{2}$ ... $j_{k}$ $\in $ $\mathcal{X}_{2}$
\end{center}

to a vertex

\begin{center}
$Wj$ $=$ $j_{1}j_{2}$ ... $j_{k}$ $j$ $\in $ $\mathcal{X}_{2}$,
\end{center}

for $j$ $=$ $1,$ $2,$ for $k$ $\in $ $\Bbb{N}.$ Then the corresponding
2-tree operator $L_{(W,Wj)}$ induced by an edge $(W,$ $Wj)$ is unitarily
equivalent to the restriction $r_{e_{j}}\mid _{\mathcal{H}_{W}}$ of the
right creation operator $r_{e_{j}}$ on $\mathcal{F}_{2}$ induced by $e_{j},$
for $j$ $=$ $1,$ $2,$ more precisely,

(5.2.12)

\begin{center}
$L_{(W,\text{ }Wj)}$ $\overset{\text{U.E}}{=}$ $r_{e_{j}}\mid _{\mathcal{H}%
_{W}}$ $:$ $\mathcal{H}_{W}$ $\rightarrow $ $\mathcal{H}_{W}(j)$
\end{center}

on $\mathcal{F}_{2}$ $=$ $H_{V},$ for all $j$ $=$ $1,$ $2,$ where

\begin{center}
$\mathcal{H}_{W}$ $=$ $p_{j_{1}}\left( \Bbb{C}^{\oplus 2}\right) $ $\otimes $
... $\otimes $ $p_{j_{k}}\left( \Bbb{C}^{\oplus 2}\right) $
\end{center}

and

\begin{center}
$\mathcal{H}_{W}(j)$ $=$ $p_{j_{1}}\left( \Bbb{C}^{\oplus }\right) $ $%
\otimes $ ... $\otimes $ $p_{j_{k}}\left( \Bbb{C}^{\oplus 2}\right) $ $%
\otimes $ $p_{j}\left( \Bbb{C}^{\oplus 2}\right) ,$
\end{center}

where $e_{1}$ and $e_{2}$ are the natural basis-vectors of $\Bbb{C}^{\oplus
2}$ and where $p_{1}$ and $p_{2}$ are the natural projections on $\Bbb{C}%
^{\oplus 2}.$
\end{lemma}

\begin{proof}
By definition and by (5.2.6), for a fixed vertex

\begin{center}
$W$ $=$ $j_{1}$ $j_{2}$ ... $j_{k}$ $\in $ $\mathcal{X}_{2}$ $=$ $V\left( 
\mathcal{T}_{2}\right) ,$ for $k$ $\in $ $\Bbb{N},$
\end{center}

the corresponding 2-tree operator $L_{(W,Wj)},$ induced by the edge $(W,$ $%
Wj),$ is unitarily equivalent to an operator $T_{j}$ on the Fock space $%
\mathcal{F}_{2},$ sending all Hilbert-space elements in $\mathcal{H}_{W}$ to
the Hilbert-space elements in $\mathcal{H}_{W}(j),$ where $\mathcal{H}_{W}$
and $\mathcal{H}_{W}(j)$ are given as in (5.2.12) above, for $j$ $=$ $1,$ $%
2. $ If an Hilbert space elements are not contained in $\mathcal{H}_{W},$
then $T_{j}$ send such elements to the zero vector $0_{\mathcal{F}_{2}}$ of $%
\mathcal{F}_{2}.$ i.e., the 2-tree operator $L_{(W,Wj)}$ is unitarily
equivalent to the operator $T_{j}$ on the vertex space $H_{V}$ $=$ $\mathcal{%
F}_{2},$ where

(5.2.13)

\begin{center}
$T_{j}\eta $ $=$ $\left\{ 
\begin{array}{ll}
r_{e_{j}}\eta =\eta \otimes e_{j} & \text{if }\eta \in \mathcal{H}_{W} \\ 
0_{\mathcal{F}_{2}} & \text{otherwise,}
\end{array}
\right. $
\end{center}

for $j$ $=$ $1,$ $2,$ where $r_{e_{j}}$ is the right creation operator
induced by $e_{j},$ where

\begin{center}
$e_{1}$ $=$ $(1,$ $0),$ and $e_{2}$ $=$ $(0,$ $1)$ in $\Bbb{C}^{\oplus 2}.$
\end{center}

Remark that the operator $T_{j}$ is nothing but the restriction $%
r_{e_{j}}\mid _{\mathcal{H}_{W}}$ of the right creation operator $r_{e_{j}}$
to the subspace $\mathcal{H}_{W}$ of $\left( \Bbb{C}^{\oplus 2}\right)
^{\otimes k}$ in $\mathcal{F}_{2}.$ Therefore, the 2-tree operator $%
L_{(W,Wj)}$ is unitarily equivalent to the restriction $r_{e_{j}}\mid _{%
\mathcal{H}_{W}}$ of $r_{e_{j}}$ on the vertex space $H_{V}$ $=$ $\mathcal{F}%
_{2},$ for $j$ $=$ $1,$ $2.$
\end{proof}

\strut By the previous lemma, we can obtain the following theorem.

\begin{theorem}
Let $r_{e_{j}}$ be the right creation operator on the Fock space $\mathcal{F}%
_{2}$ over $\Bbb{C}^{\oplus 2},$ induced by the natural basis-vector $e_{j},$
for $j$ $=$ $1,$ $2.$ Then it is unitarily equivalent to the element $R_{j}$
of the graph von Neumann algebra $M_{2}$ of the 2-regular tree $\mathcal{T}%
_{2},$ where

(5.1.14)

\begin{center}
$R_{j}$ $\overset{def}{=}$ $\underset{(W,Wj)\in E\left( \mathcal{T}%
_{2}\right) ,W\in \mathcal{X}_{2}}{\sum }$ $L_{(W,Wj)}$ $\in $ $M_{2},$ for $%
j$ $=$ $1,$ $2,$
\end{center}

where $L_{(W,Wj)}$ are the 2-tree operators induced by edges $(W,$ $Wj)$ $%
\in $ $E\left( \mathcal{T}_{2}\right) .$ Also, the right annihilation
operator $r_{e_{j}}^{*}$ is unitarily equivalent to the adjoint $R_{j}^{*}$
of $R_{j},$ and hence

(5.1.15)

\begin{center}
$r_{e_{j}}^{*}$ $\overset{\text{U.E}}{=}$ $\underset{(W,Wj)\in E(\mathcal{T}%
_{2}^{-1}),\,W\in \mathcal{X}_{2}}{\sum }$ $L_{(Wj,W)}$, for $j$ $=$ $1,$ $%
2, $
\end{center}

where $\mathcal{T}_{2}^{-1}$ means the shadow of $\mathcal{T}_{2}.$
\end{theorem}

\begin{proof}
\strut By the previous lemma, the 2-tree operators $L_{(W,Wj)}$ induced by
the edges $(W,Wj)$ are unitarily equivalent to the restrictions $%
r_{e_{j}}\mid _{\mathcal{H}_{k}}$ on the vertex space $H_{V}$ $=$ $\mathcal{F%
}_{2},$ whenever $\left| W\right| $ $=$ $k,$ for $k$ $\in $ $\Bbb{N}.$
Therefore, the right creation operators $r_{e_{j}}$ induced by the natural
basis-elements $e_{j}$ of $\Bbb{C}^{\oplus 2}$ are unitarily equivalent to
the elements $R_{j}$ of the graph von Neumann algebra $M_{2}$ of the
2-regular tree $\mathcal{T}_{2},$ in the sense of (5.1.14), on $\mathcal{F}%
_{2},$ for all $j$ $=$ $1,$ $2.$

So, the adjoint $r_{e_{j}}^{*},$ the right annihilation operators induced by 
$e_{j},$ are unitarily equivalent to the adjoints $R_{j}^{*}$ of $R_{j},$ in
the sense of (5.1.15), in $M_{2}.$
\end{proof}

\strut The above theorem shows that the products

\begin{center}
$r_{e_{j_{1}}}^{s_{1}}$ $r_{e_{j_{2}}}^{s_{2}}$ ... $r_{e_{j_{k}}}^{s_{k}}$
\end{center}

of right creation operators and right annihilation operators, for $j_{1},$
..., $j_{k}$ $\in $ $\{1,$ $2\},$ and for $s_{1},$ ..., $s_{k}$ $\in $ $\{1,$
$*\},$ is also unitarily equivalent to the elements

\begin{center}
$R_{j_{1}}^{s_{1}}$ $R_{j_{2}}^{s_{2}}$ ... $R_{j_{3}}^{s_{n}}$
\end{center}

of the graph von Neumann algebra $M_{2}$ of the 2-regular tree $\mathcal{T}%
_{2},$ where $t$ is a certain constant and $R_{j_{1}},$ ..., $R_{j_{k}}$ are
given as in (5.1.14).

Therefore, by the above theorem and by the very above discussion, we can
obtain the following corollary.

\begin{corollary}
The right Toeplitz algebra $Toep^{R}\left( \Bbb{C}^{\oplus 2}\right) ,$ in
the sense of Section 4.2, is a $C^{*}$-subalgebra of the graph von Neumann
algebra $M_{2}$ of the 2-regular tree $\mathcal{T}_{2}.$ $\square $
\end{corollary}

\strut Thus, by Section 4.2 and by the above corollary, we can obtain the
following theorem.

\begin{theorem}
\strut Let $H$ be a Hilbert space with its dimension $\dim H$ $=$ $2,$ and
let $Toep(H)$ be the Toeplitz algebra over $H.$ Then $Toep(H)$ is a $C^{*}$%
-subalgebra of the graph von Neumann algebra $M_{2}$ of the 2-regular tree $%
\mathcal{T}_{2}.$
\end{theorem}

\begin{proof}
\strut By the above corollary, the right Toeplitz algebra $Toep^{R}\left( 
\Bbb{C}^{\oplus 2}\right) $ over $\Bbb{C}^{\oplus 2}$ is a $C^{*}$%
-subalgebra of $M_{2}.$ In Section 4.2, we showed that, for any arbitrary
Hilbert space $H,$ the right Toeplitz algebra $Toep^{R}(H)$ and the Toeplitz
algebra $Toep(H)$ are anti-$*$-isomorphic from each other.

If a Hilbert space $H$ has its dimension, $\dim H$ $=$ $2,$ then it is
Hilbert-space isomorphic to $\Bbb{C}^{\oplus 2},$ and moreover, there exists
an anti-$*$-isomorphism

\begin{center}
$\Phi ^{-1}$ $:$ $Toep^{R}(H)$ $\rightarrow $ $Teop(H),$
\end{center}

where $\Phi $ is an anti-$*$-isomorphism defined in (4.2.7). Therefore, the
right Toeplitz algebra $Toep^{R}\left( \Bbb{C}^{\oplus 2}\right) $ is anti-$%
* $-isomorprhic to the Toeplitz algebra $Toep(H)$ over $H,$ whose dimension $%
\dim H$ $=$ $2.$ So, $Toep(H)$ is a $C^{*}$-subalgebra of $M_{2}.$\strut
\end{proof}

\subsection{$N$-Tree Operators for $N$ $\geq $ $2$}

Now, we will consider $N$-tree operators, where $N$ $\geq $ $2.$ In Section
5.2, we already observed the case where $N$ $=$ $2.$ So, we need to study
for $N$ $>$ $2.$ However, we can obtain the similar results as in Section
5.2, by induction. \strut In this section, we will extend the main results
of Section 5.2 to the general case where $N$ $\geq $ $2.$ i.e., we consider
the relation between generalized Toeplitz operators on the Fock space $%
\mathcal{F}_{N}$ over $\Bbb{C}^{\oplus N},$ and $N$-tree operators in the
graph von Neumann algebra $M_{N}$ of the $N$-regular tree $\mathcal{T}_{N}.$
Again, notice that our extension would be simply done by induction.

Now, we denote the process sending the vertex

\begin{center}
$i_{1}i_{2}$ ... $i_{n}$
\end{center}

\strut in the $n$-th level of $\mathcal{T}_{N}$ to the vertex

\begin{center}
$i_{1}i_{2}$ ... $i_{n}$ $j$
\end{center}

in the $(n$ $+$ $1)$-th level of $\mathcal{T}_{N}$ by $\gamma _{j},$ for all 
$n$ $\in $ $\Bbb{N}$, where $i_{1},$ ..., $i_{n}$ $\in $ $\{1,$ $2,$ ..., $%
N\}.$ i.e.,

\begin{center}
$\gamma _{j}\left( i_{1}i_{2}\text{ ... }i_{n}\right) $ $\overset{def}{=}$ $%
i_{1}$ $i_{2}$ ... $i_{n}$ $j,$
\end{center}

for all $j$ $=$ $1,$ $2,$ ..., $N,$ for all $n$ $\in $ $\Bbb{N}.$ Also, we
can understand these functions $\gamma _{j}$ generates the edges in $%
\mathcal{T}_{2}.$ i.e., the function

\begin{center}
$\gamma _{j}\left( i_{1}\text{ }i_{2}\text{ ... }i_{n}\right) $
\end{center}

can be understood as the edge

\begin{center}
$\left( i_{1}i_{2}\text{ ... }i_{n},\text{ }i_{1}i_{2}\text{ ... }%
i_{n}j\right) ,$
\end{center}

in $\mathcal{T}_{N},$ for $j$ $=$ $1,$ $2,$ ..., $N.$

Now, let $X_{N}$ be the set $\{1,$ $2,$ ..., $N\},$ and let $\mathcal{X}_{N}$
be the set

\begin{center}
$\mathcal{X}_{N}$ $=$ $\{\emptyset \}$ $\cup $ $X_{N}^{*},$
\end{center}

where $X_{N}^{*}$ is the set of all words in $X_{N},$ and $\emptyset $ means
the empty word in $X_{N}.$

We can check that

\begin{center}
$V\left( \mathcal{T}_{N}\right) $ $=$ $\mathcal{X}_{N}$ $=$ $\{\emptyset \}$ 
$\cup $ $\left\{ \gamma _{W}(\mathcal{X}_{N}):W\in X_{2}^{*}\right\} $,
\end{center}

where

\begin{center}
$\gamma $\strut $_{W}$ $=$ $\gamma _{1}\circ \gamma _{j_{2}}\circ $ ... $%
\circ $ $\gamma _{j_{k}},$
\end{center}

whenever $W$ $=$ $j_{1}$ $j_{2}$ ... $j_{k}$ $\in $ $\mathcal{X}_{N}.$
Similar to the case where $N$ $=$ $2,$ we can conclude that

\begin{center}
$H_{V}$ $=$ $l^{2}(\mathcal{X}_{N})$ $=$ $l^{2}\left( V(\mathcal{T}%
_{N})\right) $
\end{center}

is Hilbert-space isomorphic to \emph{the }(\emph{generalized})\emph{\ Fock
space }$\mathcal{F}_{N}$ $=$ $\mathcal{F}(\Bbb{C}^{\oplus N})$\emph{\ over} $%
\Bbb{C}^{\oplus N}.$

\begin{theorem}
Let $H_{V}$ be the vertex space of the graph Hilbert space $H_{N}$ of $%
\mathcal{T}_{N}.$ Then

\begin{center}
$H_{V}$ $\overset{\text{Hilbert}}{=}$ $l^{2}\left( \mathcal{X}_{N}\right) $ $%
\overset{\text{Hilbert}}{=}$ $\mathcal{F}_{N},$
\end{center}

where $\mathcal{F}_{N}$ is the Fock space over $\Bbb{C}^{\oplus N}$. $%
\square $
\end{theorem}

Let $p_{1},$ ..., $p_{N}$ be the natural projection on $\Bbb{C}^{\oplus N},$
defined by

\begin{center}
$p_{j}$ $=$ $\left( 
\begin{array}{lllllll}
0 &  &  &  &  &  & 0 \\ 
& \ddots &  &  &  &  &  \\ 
&  & 0 &  &  &  &  \\ 
&  &  & \frame{1} &  &  &  \\ 
&  &  &  & 0 &  &  \\ 
&  &  &  &  & \ddots &  \\ 
0 &  &  &  &  &  & 0
\end{array}
\right) ,$
\end{center}

having only nonzero entry $(j,$ $j)$-entry $1$, expressed by \frame{1}
above, for all $j$ $=$ $1,$ ..., $N.$

Let $e_{1},$ ..., $e_{N}$ be the natural basis elements of $\Bbb{C}^{\oplus
N},$ i.e.,

\begin{center}
$e_{j}$ $=$ $\left( 0,\text{ ...},\text{ }0,\text{ }\underset{j\text{-th}}{1}%
,\text{ }0,\text{ ..., }0\right) ,$
\end{center}

for all $j$ $=$ $1,$ ..., $N.$

\strut The above theorem characterize the vertex space $H_{V}$ in the graph
Hilbert space $H_{N}$ of the $N$-regular tree $\mathcal{T}_{N}$ by the Fock
space $\mathcal{F}_{N}$ over $\Bbb{C}^{\oplus N}.$ So, from now on, we use $%
H_{V}$ and $\mathcal{F}_{N},$ alternatively.

\begin{lemma}
\strut Let $(W,$ $Wj)$ be an edge connecting a vertex

\begin{center}
$W$ $=$ $j_{1}j_{2}$ ... $j_{k}$ $\in $ $\mathcal{X}_{N}$
\end{center}

to a vertex

\begin{center}
$Wj$ $=$ $j_{1}j_{2}$ ... $j_{k}$ $j$ $\in $ $\mathcal{X}_{N}$,
\end{center}

for $j$ $=$ $1,$ $...,$ $N,$ for $k$ $\in $ $\Bbb{N}.$ Then the
corresponding $N$-tree operator $L_{(W,Wj)}$ induced by an edge $(W,$ $Wj)$
is unitarily equivalent to the restriction $r_{e_{j}}\mid _{\mathcal{H}_{W}}$
of the right creation operator $r_{e_{j}}$ on $\mathcal{F}_{N}$ induced by $%
e_{j},$ for $j$ $=$ $1,$ $...,$ $N,$ more precisely,

\begin{center}
$L_{(W,\text{ }Wj)}$ $\overset{\text{U.E}}{=}$ $r_{e_{j}}\mid _{\mathcal{H}%
_{W}}$ $:$ $\mathcal{H}_{W}$ $\rightarrow $ $\mathcal{H}_{W}(j)$
\end{center}

on $\mathcal{F}_{N}$ $=$ $H_{V},$ for all $j$ $=$ $1,$ $...,$ $N,$ where

\begin{center}
$\mathcal{H}_{W}$ $=$ $p_{j_{1}}\left( \Bbb{C}^{\oplus 2}\right) $ $\otimes $
... $\otimes $ $p_{j_{k}}\left( \Bbb{C}^{\oplus 2}\right) $
\end{center}

and

\begin{center}
$\mathcal{H}_{W}(j)$ $=$ $p_{j_{1}}\left( \Bbb{C}^{\oplus }\right) $ $%
\otimes $ ... $\otimes $ $p_{j_{k}}\left( \Bbb{C}^{\oplus 2}\right) $ $%
\otimes $ $p_{j}\left( \Bbb{C}^{\oplus 2}\right) .$
\end{center}

$\square $
\end{lemma}

\strut By the previous lemma, we can obtain the following theorem.

\begin{theorem}
Let $r_{e_{j}}$ be the right creation operator on the Fock space $\mathcal{F}%
_{N}$ over $\Bbb{C}^{\oplus N},$ induced by the natural basis vector $e_{j},$
for $j$ $=$ $1,$ $...,$ $N.$ Then it is unitarily equivalent to the element $%
R_{j}$ of the graph von Neumann algebra $M_{N}$ of the $N$-regular tree $%
\mathcal{T}_{N},$ where

\begin{center}
$R_{j}$ $\overset{def}{=}$ $\underset{(W,Wj)\in E\left( \mathcal{T}%
_{N}\right) ,W\in \mathcal{X}_{N}}{\sum }$ $L_{(W,Wj)}$ $\in $ $M_{2},$
\end{center}

for $j$ $=$ $1,$ ..., $N,$ where $L_{(W,Wj)}$ are the $N$-tree operators
induced by edges $(W,$ $Wj)$ $\in $ $E\left( \mathcal{T}_{N}\right) .$ Also,
the right annihilation operator $r_{e_{j}}^{*}$ is unitarily equivalent to
the adjoint $R_{j}^{*}$ of $R_{j},$ and hence

\begin{center}
$r_{e_{j}}^{*}$ $\overset{\text{U.E}}{=}$ $\underset{(W,Wj)\in E(\mathcal{T}%
_{N}^{-1}),\,W\in \mathcal{X}_{N}}{\sum }$ $L_{(Wj,W)}$,
\end{center}

for $j$ $=$ $1,$ $...,$ $N,$ where $\mathcal{T}_{N}^{-1}$ means the shadow
of $\mathcal{T}_{N}.$ $\square $
\end{theorem}

\strut The above theorem shows that the products

\begin{center}
$r_{e_{j_{1}}}^{s_{1}}$ $r_{e_{j_{2}}}^{s_{2}}$ ... $r_{e_{j_{k}}}^{s_{k}}$
\end{center}

of right creation operators and right annihilation operators, for $j_{1},$
..., $j_{k}$ $\in $ $\{1,$ $...,$ $N\},$ and for $s_{1},$ ..., $s_{k}$ $\in $
$\{1,$ $*\},$ is also unitarily equivalent to the elements

\begin{center}
$R_{j_{1}}^{s_{1}}$ $R_{j_{2}}^{s_{2}}$ ... $R_{j_{3}}^{s_{n}}$
\end{center}

of the graph von Neumann algebra $M_{N}$ of the $N$-regular tree $\mathcal{T}%
_{N}$.

Therefore, by the above theorem and by the very above discussion, we can
obtain the following corollary.

\begin{corollary}
The right Toeplitz algebra $Toep^{R}\left( \Bbb{C}^{\oplus N}\right) ,$ in
the sense of Section 4.2, is a $C^{*}$-subalgebra of the graph von Neumann
algebra $M_{N}$ of the $N$-regular tree $\mathcal{T}_{N},$ for $N$ $\in $ $%
\Bbb{N}$ $\setminus $ $\{1\}.$ $\square $
\end{corollary}

\strut Thus, by Section 4.2 and by the above corollary, we can obtain the
following theorem.

\begin{theorem}
\strut Let $H$ be a Hilbert space with its dimension $\dim H$ $=$ $N,$ for $%
N $ $\in $ $\Bbb{N}$ $\setminus $ $\{1\},$ and let $Toep(H)$ be the Toeplitz
algebra over $H.$ Then $Toep(H)$ is a $C^{*}$-subalgebra of the graph von
Neumann algebra $M_{N}$ of the $N$-regular tree $\mathcal{T}_{N}.$
\end{theorem}

\begin{proof}
\strut By the above corollary, the right Toeplitz algebra $Toep^{R}\left( 
\Bbb{C}^{\oplus N}\right) $ over $\Bbb{C}^{\oplus N}$ is a $C^{*}$%
-subalgebra of $M_{N},$ for $N$ $\in $ $\Bbb{N}$ $\setminus $ $\{1\}.$ In
Section 4.2, we showed that, for any arbitrary Hilbert space $H,$ the right
Toeplitz algebra $Toep^{R}(H)$ and the Toeplitz algebra $Toep(H)$ are anti-$%
* $-isomorphic from each other.

If a Hilbert space $H$ has its dimension, $\dim H$ $=$ $N,$ then it is
Hilbert-space isomorphic to $\Bbb{C}^{\oplus N},$ and moreover, there exists
an anti-$*$-isomorphism

\begin{center}
$\Phi ^{-1}$ $:$ $Toep^{R}(H)$ $\rightarrow $ $Teop(H),$
\end{center}

where $\Phi $ is an anti-$*$-isomorphism defined in (4.2.7). Therefore, the
right Toeplitz algebra $Toep^{R}\left( \Bbb{C}^{\oplus N}\right) $ is anti-$%
* $-isomoprhic to the Toeplitz algebra $Toep(H)$ over $H.$ So, $Toep(H)$ is
a $C^{*}$-subalgebra of $M_{N}.$\strut
\end{proof}

\strut \strut \strut \strut \strut \strut 

\textbf{Acknowledgement} \emph{We have had enlightening discussions with
colleagues, Daniel Alpay, Keri Kornelson, Anna Paolucci, Erin Pearse,
Myung-Sin Song, Karen Shuman, Thomas Anderson, and Victor Vega}.

\strut

\strut \strut

\strut \textbf{References}\strut

\strut

{\small [1] \ \ A. Gibbons and L. Novak, Hybrid Graph Theory and Network
Analysis, ISBN: 0-521-46117-0, (1999) Cambridge Univ. Press.}

{\small [2]\strut \ \ \ \strut I. Cho, Fractals in Graphs, Verlag with Dr.
Muller Monographs, ISBN: 978-3-639-19447-0, (2009), VDM publisher.}

{\small [3] \ \ I. Cho, Hyponormality of Toeplitz Operators with
Trigonometric Polynomial Symbols, Master Degree Thesis, (1999) Sungkyunkwan
Univ.}

{\small [4] \ \ \strut I. Cho, Graph Groupoids and Partial Isometries, ISBN:
978-3-8383-1397-9, (2009) LAP Publisher.\strut }

{\small [5] \ \ \strut I. Cho, and P. E. T. Jorgensen, }$C^{*}${\small %
-Subalgebras Generated by Partial Isometries, JMP, DOI: 10.1063/1.3056588,
(2009).}

{\small [6] \ \ \strut I. Cho, and P. E. T. Jorgensen, }$C^{*}${\small %
-Subalgebras Generated by a Single Operator in }$B(H)${\small , ACTA Appl.
Math., 108, (2009) 625 - 664.}

{\small [7] \ \ \strut I. Cho, and P. E. T. Jorgensen, Application of
Automata and Graphs: Labeling Operators in Hilbert Space II, J. Math. Phy.,
DOI: 10.1063/1.3141524, (2009).}

{\small [8] \ \ I. Cho, and P. E. T. Jorgensen, Operators Induced by Graphs,
Preprint, (2010) Submitted to Central Euro. J. Math..}

{\small [9] \ \ \strut P. D. Mitchener, }$C^{*}${\small -Categories,
Groupoid Actions, Equivalent KK-Theory, and the Baum-Connes Conjecture,
arXiv:math.KT/0204291v1, (2005), Preprint.}

{\small [10] R. Gliman, V. Shpilrain and A. G. Myasnikov (editors),
Computational and Statistical Group Theory, Contemporary Math, 298, (2001)
AMS.}

{\small [11] F. Radulescu, Random Matrices, Amalgamated Free Products and
Subfactors of the $C^{*}$- Algebra of a Free Group, of Noninteger Index,
Invent. Math., 115, (1994) 347 - 389.}

{\small [12] P. R. Halmos. Hilbert Space Problem Book (2-nd Ed), ISBN:
0-387-90685-1, (1982) Springer-Verlag.}

{\small [13] T. Yosino, Introduction to Operator Theory, ISBN:
0-582-23743-2, (1993) Longman Sci. \& Tech.}

{\small [14] P. E. T. Jorgensen, Essential Self-Adjointness of the
Graph-Laplacian, J. Math. Phy., 49, no. 7, (2008) 073510, 33.}

{\small [15] P. E. T. Jorgensen, E. J. Pearse, Operator Theory of Electrical
Resistance Networks, Book Manuscript, http://arxive.org/abs/0806.3881, (2008)%
}

{\small [16] C. T. Zamfirescu, An Infinite Family of Planar non-Hamiltonian
Bihomogeneousely Traceable Oriented Graphs, Graphs Combin., 26, no. 1,
(2010) 141 - 146.}

{\small [17] P. Dorbec, and S. Gravier, Paired-Domination in Subdivided
Star-Free Graphs, Graphs Combin., 26, no. 1, (2010) 43 - 49.}

{\small [18] S-C Chang, and R. Shrock, Weighted Graph Colorings, J. Stat.
Phys., 138, no. 1-3, (2010) 496 - 542.}

{\small [19] M. Ferrara, M. S. Jacobson, and A. Harris, Cycle Lengths in
Hamiltonian Graphs with a Pair of Vertices Having Large Degree Sum, Graphs.
Combin., 26, no. 2, (2010) 215 - 223.}

{\small [20] I. Raebun, M. Tomforde, and D. P. Williams, Classification
Theorems for the }$C^{*}${\small -Algebras of Graphs with Sinks, Bull.
Austral. Math. Soc., 70, no. 1, (2004) 143 - 161.}

{\small [21] N. J. Fowler, and I. Raeburn, The Toeplitz Algebra of a Hilbert
Bimodule, Indiana Univ. Math. J., 48, no. 1, (1999) 155 - 181.}

{\small [22] A. Kumjian, D. Pask, and I. Raeburn, Cuntz-Frieger Algebras of
Directed Graphs, Pacific J. Math., 184, no. 1, (1998) 161 - 174.}

{\small [23] R. R. Coifman, Y. Shkolnisky, and F. J. Sigworth, Graph
Laplacian Tomography From Unknown Random Projections, IEEE Trans. Image
Process., 17, no. 10, (2008) 1891 - 1899.}

{\small [24] R. D. Lazarov, and S. D. Margenov, CBS Constants for Multilevel
Splitting of Graph Laplacian and Application to Preconditioning of
Discontinuous Systems, J. Complexity, 23, (2007) 4 - 6.}

{\small [25] M. Hein, J-Y Audibert, and U. von Luxberg, Graph Laplacians and
Their Convergence on Random Neighborhood Graphs, J. Mach. Learn. Res., 8,
(2007) 1325 - 1368.}

{\small [26] W. F. Stinespring, Positive Functions on }$C^{*}${\small %
-Algebras, Proc. Amer. Math. Soc., 6, (1955) 211 - 216.}

{\small [27] F. J. Murray, and J. von Neumann, On Rings of Operators II,
Trans. Amer. Math. Soc., 41, no. 2, (1937) 208 - 248.}

{\small [28] N. D. Popova, Yu. S. Samoulenko, and A. V. Strelets, On Coxeter
Graph Related Configurations of Subspaces of a Hilbert Space, Modern
Analysis and Applications, The Mark Krein Contenary Conference, vol. 1:
Operator Theory and Related Topics, Oper. Theo. Adv. Appl., 190, (2009) 429
- 450.}

{\small [29] R. D. Mauldin, T. Szarek, and M. Urbanski, Graph Directed
Markov Systems on Hilbert Spaces, Math. Proc. Cambridge Philos. Soc., 147,
no. 2, (2009) 455 - 488.}

{\small [30] S. Salimi, Continuous-time Quantum Walks on Star Graphs, Ann.
Phy., 324, no. 6, (2009) 1185 - 1193.}

{\small [31] V. Kostrykin, J. Potthoff, and R. Schrader, Contraction
Semigroups on Metric Graphs, Analysis on Graphs and its Applications, Proc.
Sympos. Pure Math., vol. 77, Amer. Math. Soc., (2008) 423 - 458.}

{\small [32] P. Kuchment, Quantum Graphs: An Introduction and a Brief
Survey, Analysus on Graphs and its Applications, Proc. Sympos. Pure Math.,
vol. 77, Amer. Math. Soc., (2008) 291 - 312.}

{\small [33] K. R. Davidson, S. C. Power, and D. Yang, Atomic
Representations of Rank 2 Graph Algebras, J. Funct. Anal., 255, no. 4,
(2008) 819 - 853.}

{\small [34] M. Baker, and R. Rumely, Harmonic Analysis on Metrized Graphs,
Canad. J. Math., 59, no. 2, (2007) 225 - 275.}

{\small [35] R. V. Kadison, and J. R. Ringrose, Fundamentals of the Theory
of Operator Algebra I, Grad. Stud. Math., vol. 15, Amer. Math. Soc., (1997) }

{\small [36] W. Arveson, An Invitation to }$C^{*}${\small -Algebras, Grad.
Texts. Math., no. 39, (1976) Springer-Verlag.}

\end{document}